\documentclass[12pt]{amsart}

\usepackage{amsmath,amssymb,latexsym}
\usepackage{dsfont}
\usepackage[T1]{fontenc}
\usepackage{enumerate}

\usepackage[a4paper,top=3cm, bottom=3cm, inner=3.5cm,outer=2.7cm,nofoot,headsep=1.5cm]{geometry}
\def\Z{\mathbb Z}
\def\bG{\mathbb G}
\def\bH{\mathbb H}
\def\M{\widetilde{\mathfrak M}_c}
\def\Mn{\mathfrak M_c}

\def\F{\mathcal F}

\def\G{\mathcal G}
\def\S{\mathcal S}

\def\mH{\mathcal H}

\def\[{\tilde{[}}
\def\]{\tilde{]}}
\def\t{\widetilde}
\def\b{\mathbb}
\def\ls{\lesssim}

\def\tg{\t \gamma}
\def\supp{\operatorname{supp}}
\def\FC{\operatorname{FC}}
\def\FN{\operatorname{FN}}

\def\Ind#1#2{#1\setbox0=\hbox{$#1x$}\kern\wd0\hbox to 0pt{\hss$#1\mid$\hss}
\lower.9\ht0\hbox to 0pt{\hss$#1\smile$\hss}\kern\wd0}

\def\Notind#1#2{#1\setbox0=\hbox{$#1x$}\kern\wd0\hbox to 0pt{\mathchardef
\nn="3236\hss$#1\nn$\kern1.4\wd0\hss}\hbox to 0pt{\hss$#1\mid$\hss}\lower.9\ht0
\hbox to 0pt{\hss$#1\smile$\hss}\kern\wd0}

\theoremstyle{plain}
\newtheorem{theorem}{Theorem}[section]
\newtheorem{prop}[theorem]{Proposition}
\newtheorem{fact}[theorem]{Fact}
\newtheorem{lemma}[theorem]{Lemma}
\newtheorem{cor}[theorem]{Corollary}
\newtheorem*{claim}{Claim}

\theoremstyle{definition}
\newtheorem{defn}[theorem]{Definition}
\newtheorem{remark}[theorem]{Remark}

\newtheorem{expl}{Example}

\newtheorem*{cla}{Claim}

\newtheorem*{nota}{Notation}

\title{Almost group theory}
\author{Nadja Hempel}

\begin{document}

\begin{abstract}The notion of almost centralizer and almost commutator are introduced and basic properties are established. They are used to study $\M$-groups, i.\ e.\ groups for which every descending chain of centralizers each having infinite index in its predecessor stabilizes after finitely many steps. The Fitting subgroup of such groups is shown to be nilpotent and theorems of Fitting and Hall for nilpotent groups are generalized to ind-definable almost nilpotent subgroups of $\M$-groups.  \end{abstract}

\maketitle

\section{Introduction}

Groups in which every descending chain of centralizers stabilizes after finitely many steps, so called $\Mn$-groups, have been of great interest to both group and model theorist. They have been studied by  Altinel and Baginski \cite{AltBag}, Bryant \cite{Bry}, Bryant and Hardley \cite{Bry2},  Derakhshan and Wagner \cite{DerWag},  Poizat and Wagner \cite{PoiWag}. In the field of model theory they appear naturally as definable groups in stable and o-minimal theories. Passing to groups definable in simple theories or even more general rosy theories, we obtain a weaker chain condition, namely any chain of centralizers, each having infinite index in its predecessor, stabilizes after finitely many steps. We want to study group for which any definable quotient
satisfies this chain condition which we call $\M$-groups. As mention above, examples are definable groups in any simple theory, such as simple pseudo-finite groups and groups definable in the theory of perfect bounded PAC-fields, as well as groups definable in rosy theories. 
A useful notion in this context is the FC-centralizer of a subgroup: For a subgroup $H$ of a group $G$, the FC-centralizer contains all elements whose centralizer has finite index in $H$. These subgroups were introduced by Haimo in \cite{Hai}. Defining a suitable notion of these objects regarding $A$-invariant subgroups of $G$ as well as the notion of  an \emph{almost commutator}, another useful object to study $\M$-groups, and establishing their basic properties is the main part of Section \ref{Sec_ACAC}.

In Section \ref{sec_Mcgroups} we start working in $\M$-groups. One of its crucial properties is that the iterated FC-centralizer of any subgroup is definable. In particular, this allows us to have the same connection between almost commutator and almost centralizers as in the ordinary case. Additionally, we show that any abelian, solvable or nilpotent subgroup of an $\M$-group is almost contained (up to finite index) in a definable subgroup with the same property generalizing these results from groups definable in simple theories by Milliet \cite{cm_defenv} to this more general context.

Afterward, we analyze the Fitting subgroup of an $\M$-group   (Section \ref{Sec_Fit}). For a group $G$, the Fitting subgroup is the group generated by all normal nilpotent subgroups of $G$. Note that this subgroup is always normal in $G$ but might not be nilpotent. For $\Mn$-groups, nilpotency of the Fitting subgroup was shown by Bryant \cite{Bry} for periodic groups, by Poizat and Wagner \cite{PoiWag} in the stable case and in general by Derakhshan  and Wagner \cite{DerWag}. Furthermore, it has been recently generalized by Palac\'in and Wagner \cite{PalWag} to groups type-definable in simple theories. One of the main ingredient other than the chain condition on centralizers, is that any nilpotent subgroup has a definable envelop up to finite. As mentioned before, we establish this  result for $\M$-groups in Section \ref{sec_Mcgroups} which enables us to prove nilpotency of the Fitting subgroup for $\M$-groups.

In the last section, we study subgroups of $\M$-groups which are almost nilpotent or FC-nilpotent, a notion which was also introduced by Haimo in \cite{Hai}: a subgroup $H$ of $G$ is FC-nilpotent, if there is a sequence  $\{1\}= H_0 < H_1 < \dots < H_n = H$ of normal nilpotent subgroups of $G$ such that $H_{i+1} / H_i$ is in the FC-center of $G/H_i$. Using the new notion of almost commutators and its properties in $\M$-groups, we are able to generalize 
the nilpotency criteria of Hall ($G$ is nilpotent if one can find a nilpotent subgroup $N$ such that $G$ modulo the derived subgroup of $N$ is nilpotent as well) to ind-definable almost nilpotent subgroups of $\M$-groups.

\begin{nota}
For two subgroups $H$ and $N$ of $G$ such that $H$ normalizes $N$ and $g$ an element of $G$, we let $C_H(\{gN\})$ the subgroup of $H$ which contains all elements $h$ in $H$ such that $h g \cdot N = g h \cdot N$. Moreover, we denote by $H/N$ the quotient group $H/(H\cap N)$.
\end{nota}

\section{Almost centralizers and commutators}\label{Sec_ACAC}

For this section we fix a group $G$ which is not necessary a pure group, a parameter set $A$ and assume that $G$ is \emph{$|A|$-saturated} and \emph{$|A|$-strongly homogeneous}. We want to study subgroups of $G$ which are stabilized by all automorphisms which fix point-wise $A$. From the model theoretic point of view, as $G$ is sufficiently homogeneous, this corresponds to subgroups which are set-wise a union of realizations of a family of types over $A$. Hence an $A$-invariant subgroup $H$ of $G$, as a set, has a canonical interpretation in any elementary extension $\G$ of $G$ (the set of realization in $\G$ of the family of types which define $H$). By saturation, this set forms again a subgroup of $\G$ and we denote it by $H(\G)$.

\begin{remark}
Let $H$ and $K$ be two $A$-invariant subgroups of $G$ and $\G$ be an elementary extension of $G$. By saturation and homogeneity, the normalizer of $H$ in $K$ is the trace of the normalizer of $H(\G)$ in $K(\G)$. Moreover, if $H$ is normalized by $K$, the group $H(\G)$ is normalized by $K(\G)$.
\end{remark}


\subsection{Preliminaries} \

The following subgroups were introduced by Haimo in \cite{Hai}.

\begin{defn}
Let $H$ be a group and $L$, $K$ and $N$ be three subgroup of $H$ such that $N$ is normalized by $L$. We define:
\begin{itemize}
\item The \emph{FC-centralizer} of $L$ in $K$ modulo $N$:
$$\FC_K(L/N) = \{ g \in N_K(N): [L: C_L(g/ N)] \mbox{ is finite} \}$$
\item The \emph{FC-center} of $H$:
$$\FC(H) = \FC _H(H)$$
\item The \emph{$n$th  FC-center} of $H$ inductively on $n$ as the following:
$$\FC^0(H) = \{1\} \mbox{ and } \FC^{n+1}(H)= \FC_H(H/\FC(H)_{n})$$
\item The \emph{FC-normalizer} of $H$ in $G$:
$$\FN_H(L)= \{ g \in H:  [L^g : L\cap L^g] \mbox{ and }[ L :L\cap L^g] \mbox{ are finite} \}$$
%
%
\end{itemize}
\end{defn}

We generalize this notion to a suitable notion of these objects regarding $A$-invariant subgroups of $G$. For two such groups, we have two options regarding the index of one in the other: it is either \emph{bounded}, i.\ e.\ it does not grow bigger than a certain cardinal while enlarging the ambient model, or for any given cardinal $\kappa$ we can find an ambient model such that the index is larger than $\kappa$. Then we say that the index is \emph{unbounded}. This leads to the below definition.

\begin{defn}
Let $H$ and $K$ be two $A$-invariant subgroups. We say that $H$ is \emph{almost contained} in $K$, denoted by $H \ls K$ if the index of $H\cap K$ is bounded in $H$. We say that $H$ and $K$ are \emph{commensurate}, denoted by $H \sim K$ if $H$ is almost contained in $K$ and $K$ is almost contained in $H$.

In the case of arbitrary subgroup $H$ and $K$, we say that $H$ is \emph{almost contained} in $K$, also denoted by $H \ls K$ if the index of $H\cap K$ is finite in $H$. We say that $H$ and $K$ are \emph{commensurable}, denoted by $H \sim K$ if $H$ is almost contained in $K$ and $K$ is almost contained in $H$.
\end{defn}

Observe that being almost contained is a transitive relation and being commensurate is an equivalence relation among ($A$-invariant) subgroups of $G$ and that we have the following property:

\begin{lemma}\label{Lem_ProdCon}
Let $H$, $K$, and $L$ be three $A$-invariant subgroups of $G$ such that $H$ normalizes $K$. If $H\ls L$ and $K \ls L$ then $HK \ls L$.
\end{lemma}

\proof
We assume that $G$ is sufficiently saturated. By assumption, we have that the index of $L\cap H$ in $H$ as well as the index of $L \cap K$  in $K$ are bounded by some cardinal $\kappa_H$ and $\kappa_K$ respectively which are smaller than $(2^T)^+$. Take $I_H=\{h_i : i < \kappa_H\}$ and $I_K=\{k_i : i < \kappa_K\}$ representatives of the cosets of $L\cap H$ in $H$ and respectively of $L \cap K$  in $K$. Then the set $I_H\cdot I_K$ has at most size $2^T$ and as $H$ normalizes $K$, it contains a set of representatives of the cosets of $L\cap (HK)$ in $HK$, hence the index of $L\cap (HK)$ in $HK$ is bounded in any elementary extension of $G$ and whence $HK \ls L$.
\qed


\begin{defn}
Let $H$, $K$ and $N$ be three $A$-invariant subgroup of $G$ such that $N$ is normalized by $H$. We define:
\begin{itemize}
\item The \emph{almost-centralizer} of $H$ in $K$ modulo $N$:
$$\t C_K(H/N) = \{ g \in N_K(N): H \sim C_H(g/ N) \}$$
\item The \emph{almost center} of $H$:
$$\t Z(H) = \t C_H(H)$$
\item The \emph{almost normalizer} of $H$ in $G$:
$$\t N_G(H)= \{ g \in G:  H^g \sim H\cap H^g \mbox{ and } H \sim H\cap H^g \}$$
\end{itemize}
\end{defn}

\begin{remark}
These are all $A$-invariant subgroups of $G$ and each of them contain their ordinary version, i.\ e.\ $C_G(H) \leq \t C_G(H)$, $Z(G) \leq \t Z(G)$, and $N_G(H) \leq \t N_G(H)$. Note that if $H$ is definable, bounded can be replaced by finite and these definition coincide with the definition of the FC-centralizer, FC-center and FC-normalizer of $H$. Additionally, the almost centralizer $\t C_G(H)$ of $H$ is fixed by all automorphism which fix $H$ and hence a characteristic subgroup. Inductively, this gives the following well defined definition of the iterated almost centralizers.
\end{remark}

\begin{defn}
Let $H$ be an $A$-invariant subgroup of $G$, then
\begin{itemize}
\item The \emph{$n$th almost centralizer} of $H$ in $G$ inductively on $n$ as the following:
$$\t C^0_G(H)= \{1\} \mbox{ and } \t C^{n+1} _G(H)= \t C_G(H/\t C_G(H)_{n})$$

\item The \emph{$n$th almost center} of $G$ inductively on $n$ as the following:
$$\t Z_n(G) = \t C^n_G(G) $$
\end{itemize}
\end{defn}


\subsection{Symmetry and the approximate three-subgroup-lemma}\label{sec_symtsl} \


First we introduce the notion of $A$-ind-definable subgroup which are special $A$-invariant subgroup. It is a model theoretic notion which generalize type-definable subgroups.
\begin{defn}
An $A$-\emph{ind-definable} subgroup of a group $G$ is the union of a directed system of $A$-type-definable subgroups of $G$.
\end{defn}

Observe that for two subgroups $H$ and $K$ of a group $G$, we have trivially that $H \leq C_G(K)$ if and only if $K \leq C_G(H)$. In the case of almost centralizers and almost containment,  we will see that this is not true in general. Nevertheless, we obtain the same symmetric condition replacing the centralizer by the almost centralizer and containment by almost containment for ind-definable subgroups. In the case of ind-definable subgroups of groups hyperdefinable in a simple theory, this was proven by Palacin and Wagner  in \cite{PalWag}. 

We use the following fact due to B. Neumann.

\begin{fact}\label{Fact_Neu}
A group cannot be covered by finitely many translates of subgroups of infinite index.
\end{fact}

\begin{theorem}\label{Thm_sym}
Let $H$ and $K$ be two $A$-ind-definable subgroups of $G$ and let $N$ be a subgroup of $G$ which is a union of definable sets. Suppose $N$ is normalized by $H$ and $K$. Then $H\ls \t C_G(K/N)$ if and only if $K \ls \t C_G(H/N)$.
\end{theorem}

\proof
We may assume that $G$ is enough saturated, that $H$ and $K$ are ind-definable over the empty set and that $N$ is the union of $\emptyset$-definable sets. Let $\kappa$ be equal to $2^{|T|}$ and $N$ be equal to $\bigcup_{\beta \in \Delta} N_\beta$ with $N_\beta$ definable and $\Delta$ some index set.
Suppose $H$ is almost contained in $\t C_G(K/N)$. Then $H_0$, defined as the intersection of $H$ and $\t C_G(K/N)$, has bounded index in $H$, and for all elements $h$ in $H_0$ we have that $K$ is almost contained in its centralizer, i.\ e.\
$$\forall h \in H_0: K \ls C_K(h/N) (\ast)$$
Assume towards a contradiction that $K$ is not almost contained in $\t C_G(H/N)$. Thus, there is a set of representatives $\{k_i:i \in I\}$ in $K$ with $I$ greater than $(2^{\kappa})^+$ of different cosets of $\t C_K(H/N)$ in $K$. As $H$ is the union of type-definable subgroups $H_\alpha$ with $\alpha$ in a small index set $\Omega$, for every $i$ different than $j$ in $I$, there is $\alpha_{(i,j)}$ in $\Omega$ such that the centralizer of the element $k^{-1}_i k_j/N$ has unbounded index in $H_{\alpha_{(i,j)}}$. By Erdos-Rado, we can find a subset $I_0$ of $(2^{\kappa})^+$ of cardinality $\kappa^+$  and $\alpha$ in $\Omega$ such that for all indexes of the set $\{(i,j): i \not= j, i,j \in I_0\}$, we have that $\alpha_{(i,j)}$ is equal to $\alpha$. Now, the centralizers of any element in $\{ k^{-1}_i k_j/ N : i \not = j, i,j \in I_0\}$ has infinite index in the subgroup $H_\alpha$. Hence, by Fact \ref{Fact_Neu}, $H_\alpha$ can not be covered by finitely many translates of a finite union of these centralizers. So the partial type below is consistent.
$$\pi(x_n)_{ n \in \kappa} = \{ [x_n^{-1} x_m,  k^{-1}_i k_j]\not \in N_\beta: n\not = m \in \kappa, i\not = j \in I_0\}_{ \beta \in \Delta} \cup \{ x_n \in H_\alpha \}_{n\in \kappa}$$
As $G$ is sufficiently saturated, one can find a tuple $\bar h$ in $G$ which satisfies $\pi(\bar x)$. Fix two different elements $n$ and $m$ in $\kappa$. Then, we have that $k^{-1}_i k_j \not \in C_K (h_n^{-1} h_m/N )$ for all $i$ different than $j$. Thus, $C_K (h_n^{-1} h_m/ N )$ has unbounded index in $K$ witnessed by $( k_j : j \in I_0)$. As the group $H_0$ which was defined in the beginning of the proof has bounded index in $H$ and $H_\alpha$ is a subgroup of $H$, there is some tuple $(n,m)$ as above for which the element $h_n^{-1} h_m$ is in $H_0$ which contradicts $(\ast)$.
%
%
%
\qed



\begin{cor}
Let $H$ and $K$ be two definable subgroup of $G$. Then we have that
$$[H: \t C_H (K)] < \omega \mbox{ if and only if } [K: \t C_K (H)] < \omega$$
\end{cor}

\proof
By the previous theorem we have that $[H: \t C_H (K)]$ is bounded  if and only if $ [K: \t C_K (H)]$ is bounded. Thus it is enough to show that in these cases bounded index implies finite index. 

As $K$ and $H$ are definable we have that  $\t C_H (K)$ is the union of definable sets, namely:
$$ \t C_H (K) = \bigcup_{i \in \omega} H_d$$
where
$$H_d = \{ h \in H : [K : C_K(h) ] < d\}.$$
Suppose that the index $[H: \t C_H (K)]$  is at least countable. Then we have that for any cardinal $\kappa$ the following type is consistent:
$$\pi( x_i : i \in \kappa) = \{x_i \in H\} \cup \{ x_i^{-1}x_j \not\in H_d : i \neq j, d \in \omega\}$$
This contradicts that the index is bounded and thus yields the result.
\qed

In the general context, we may asked if symmetry holds for FC-centralizers. We will give a positive answer in the case that the ambient group satisfies the following chain condition on centralizer:
\vspace{8pt}
$$\emph{For every sequence $(a_i: i \in \omega)$ in $H$ there exists a natural number $n$ such that}$$
\vspace{-18pt} $$\emph{$\bigcap_{i=0}^n C_G(a_i)$ is contained in $C_G(a_j)$ for all natural numbers $j$.}$$

These groups are called \emph{$\Mn$-group}. Afterwards, we give a counter-example which shows that it does not hold in general.
\begin{prop}\label{Prop_symMc}
Let $H$ be an $\Mn$-group and $K$, $L$ and $N$ be subgroups of $H$ such that $N$ is normalized by $K$ and by $L$. Then $K\ls  \FC_H(L/N)$ if and only if $L \ls \FC_H(K/N)$.
\end{prop}

\proof
We may assume that $N$ is trivial. Suppose that $K\ls \FC_H(L)$. Then, the group $K'$ defined as the intersection of $K$ with $ \FC_H(L)$ is obviously contained in $ \FC_H(L)$ and has finite index in $K$. Note that by the latter the FC-centralizer of $K'$ in $H$ is equal to the one of $K$ in $H$. Since $H$ is an $\Mn$-group, we can find elements $k_0, \dots, k_n$ in $K'$ such that $C_L(K')$ is equal to the intersection of the centralizers of the $k_i$'s. As each $k_i$ is contained in the FC-centralizer of $L$ in $H$, this intersection and hence $C_L(K')$ has finite index in $L$. In other words, $L$ is almost contained in $ C_H(K')$. The latter group is trivially contained in $\FC_H(K')$ which  coincides with $\FC_H(K)$ as mentioned before. This finishes the proof.
\qed

The next example was suggested by F. Wagner.
\begin{expl}
Let $G$ be a finite non-commutative group, $K$ be $\prod_\omega G$ and $H$ be $\bigoplus_\omega G$. By the support of an element $(k_i: i \in \omega)$ in $K$ we mean the set of indexes $i \in \omega$ such that $k_i$ is non trivial. As any element of $H$ has finite support and $G$ is finite,  and element of $H$ has finitely many conjugates, namely at most $|G| \cdot |\supp(h)|$ many.  Thus its centralizer has finite index in $K$. Hence $H$ is contained in the FC-centralizer of $K$. On the other hand, fix an element $g$ of $G$ which is not contained in the center of $G$. Let $\bar k_0$ be the neutral element of $K$ and for $n\geq 1$ we define:

$\bar k_n = (k_i)_{i \in \omega}\mbox{ such that }\begin{cases} k_i = g &\mbox{if } i \equiv 0  \pmod{n} \\
k_i =1 & \mbox{else }  \end{cases}$

For any two different natural numbers $n$ and $m$, we have that the element $\bar k_n^{-1}\bar k_m$ is a sequence of the neutral element of $G$ and infinitely many $g$'s or $g^{-1}$'s. Now, we can choose an element $h$ in $G$ that does not commute with $g$ and for any $j$ in the support of $\bar k_n^{-1}\bar k_m$ we define the following elements of $H$:

$\bar l_j = (l_i)_{i \in \omega}\mbox{ such that }\begin{cases} l_i = h &\mbox{if }i=j \\
l_i =1 & \mbox{else }  \end{cases}$

These elements witness that the set of conjugates $(\bar k_n^{-1}\bar k_m)^H$, with $n$ different than $m$, is infinite and hence the $\bar k_n$'s  are representatives of different cosets of $\t C_K(H)$ in $K$. Thus $K$ is not almost contained in the FC-centralizer of $H$ in $K$ which contradicts symmetry.
\end{expl}

For subgroups $H$, $K$, $L$ and $M$ of some group we have that
$$[H, K, L] \leq M \mbox{ and }[K, L, H] \leq M  \mbox{ imply } [L, H, K] \leq M ,$$
which is known as the three subgroup lemma. We want to generalize this result to our framework. As we have not yet introduce an approximated version of the commutator, observe that, if $H$, $K$, and $L$ normalize each other, $[H, K, L] =1$ if and only if $H \leq C_G(K/ C_G(L))$. Thus we may translate this result as stated below.
$$H \leq C_G(K/ C_G(L)) \mbox{ and }K \leq  C_G(L/ C_G(H)) \mbox{ imply } L \leq  C_G(H/ C_G(K)).$$
%
This statement, replacing all centralizers and containment by almost centralizers and almost containment, can be deduced from the lemma proven below in the case of  ind-definable subgroups.

\begin{lemma}
Let $H$, $K$ and $L$ be three $A$-ind-definable subgroups of $G$ such that $L$ is normalized by $K$. Then the following is equivalent:
\begin{itemize}
\item $H \not\ls \t C_G(K/\t C_G(L))$;
\item for any cardinal $\kappa$, there exists an extension $\mathcal G$ of $G$ and elements $(h_i : i \in \kappa)$ in $H(\mathcal G)$, $(k_n: n \in \kappa)$ in $K(\mathcal G)$ and $(l_s : s\in \kappa)$ in $L(\mathcal G)$ such that
$$ [[h_i^{-1}h_j,k_n^{-1}k_m],l_s^{-1}l_t] \not = 1\ \ \ \forall  i,j,n,m,s,t\in \kappa,\ i \not= j, n \not= m, s \not=t.$$
\end{itemize}
\end{lemma}

\proof
We suppose that $G$ is sufficiently saturated, that $H$, $K$, and $L$ are ind-definable over the empty set and that $H \not\ls \t C_G(K/\t C_G(L))$. As $L$ is ind-definable, it is equal to a bounded union of type-definable subgroups $L_\alpha$ with $\alpha$ in some small index set  $\Omega_L$.

\begin{cla}
$H \not\ls \t C_G(K/\t C_G(L))$ if and only if there exists some $\alpha$ in $\Omega_L$ such that $H \not\ls \t C_G(K/\t C_G(L_\alpha))$.
\end{cla}

\proof
For any $\alpha$ in $\Omega_L$, the subgroup $\t C_G(L_\alpha)$ contains $\t C_G(L)$ and therefore, if $[K: C_K(g /\t C_G(L))]$ is bounded we have as well that $[K: C_K(g / \t C_G(L_\alpha))]$ is bounded. Hence $\t C_G(K/\t C_G(L_\alpha))$ contains $\t C_G(K/\t C_G(L))$ and whence the implication from right to left holds. Now let $\delta$ be equal to $(2^{|T|})^+$ and suppose that $H \not\ls \t C_G(K/\t C_G(L))$.
Then one can find $(h_i: i \in (2^\delta)^+)$ in $H$ such that for $i$ different than $j$, the element $h_i^{-1} h_j$ does not belong to $ \t C_G (K / \t C_G(L))$ or equivalently $K \not \ls C_K( h_i^{-1}h_j /C_G(L))$. Let $i$ and $j$ be two different arbitrary ordinal numbers less than $(2^\delta)^+$. Then there exists $(k^{(i,j)}_n : n \in  (2^\delta)^+)$ such that
$$\left[h_i^{-1} h_j , (k^{(i,j)}_n)^{-1} k^{(i,j)}_m\right] \not \in \t C_G(L).$$
Note that if the centralizer of some element $g$ in $G$ has unbounded index in $L$ there exists also an $\alpha$ in $\Omega_L$ such that $C_{L_\alpha}(g)$ has unbounded index in $L_\alpha$. Thus the almost centralizer of $L$ is $G$ is the intersection of the almost centralizers of the $L_\alpha$'s in $G$.  So for any $n$ different than $m$ one can find $\alpha^{(i,j)}_{(n,m)}$ such that
$$\left[h_i^{-1} h_j , (k^{(i,j)}_n)^{-1} k^{(i,j)}_m\right] \not \in \t C_G\left(L_{\alpha^{(i,j)}_{(n,m)}}\right).$$
By Erdos-Rado there exists $I_{(i,j)}$ of $(2^\delta)^+$ of cardinality at least $\delta^+$ and $\alpha_{(i,j)}$ such that for all $n$ and $m$ in $I_{(i,j)}$, we have
$$\left[h_i^{-1} h_j , (k^{(i,j)}_n)^{-1} k^{(i,j)}_m\right] \not \in \t C_G\left(L_{\alpha_{(i,j)}}\right).$$
As $i$ and $j$ were arbitrary, we have that for any $i$ different than $j$, $$K \not \ls C_K( h_i^{-1}h_j/ \t C_G(L_{\alpha_{(i,j)}}))$$  for some $\alpha_{(i,j)}$ in $\Omega$. Again by Erdos-Rado there exists a subset $I$ of  $(2^\delta)^+$ of cardinality at least $\delta^+$ and $\alpha$ such that $\alpha_{(i,j)}$ is equal to $\alpha$ for $i$ different than $j$ in $I$ and thus for any such tuples we have $K \not \ls C_K( h_i^{-1}h_j /\t C_G(L_\alpha))$. Hence the elements $h_i^{-1}h_j$ do not belong to $\t C_G(K / \t C_G(L_\alpha))$ and whence witness that $H \not\ls \t C_G(K/\t C_G(L_\alpha))$.
\qed$_{\operatorname{claim}}$

By the claim, we may assume that $L$ is a type-definable group. Hence any relatively definable subgroup of $L$ has either finite or unbounded index, whence the group $\t C_G(L)$ is equal to the union of the following definable sets
$$S_{\phi, d}= \{ g \in G: \forall l_0, \dots, l_d \models \phi(x) \bigvee_{i\not= j} l_i^{-1} l_j \in C_G( g) \},$$
where $\phi(x)$ ranges over the formulas in the type $\pi_{L}(x)$ which defines $L$ and $d$ over all natural number.

Now, let $\kappa$ be any cardinal and to easer notation we let $\lambda$ be the Cartesian product of $\kappa$ with itself leaving out the diagonal. As in the proof of Theorem \ref{Thm_sym} we can find $(h_i : i \in \kappa)$ in $H$ a set of representatives of cosets of $\t C_H(K/\t C_G(L))$ in $H$. Hence, the centralizer of the elements $h^{-1}_i h_j/ \t C_G(L)$ for $(i,j)$ in $\lambda$ has infinite index in $K$ modulo $\t C_G(L)$. Whence, by Fact \ref{Fact_Neu}, the group $K/ \t C_G(L)$ can not be covered by finitely many translates of a finite union of these centralizers.
So the partial type below is consistent.
\begin{eqnarray*}
\pi(x_n : n \in \kappa) & = &\{ [h^{-1}_i h_j, x_n^{-1} x_m]  \not\in S_{\phi, d} : (n,m),(i,j)\in \lambda, \, d \in \omega, \phi \in \pi_L\} \\
&&\cup \{ x_n \in K : n\in \kappa\}
\end{eqnarray*}
Take $\bar k$ which satisfies $\pi(\bar x)$. By construction we have that $[h_i^{-1}h_j,k_n^{-1}k_m] \not\in \t C_G(L)$. Hence, $L \not\ls C_L([h_i^{-1}h_j,k_n^{-1}k_m])$. So $L$ cannot be covered by finitely many translates of a finite union of these centralizers. So the partial type below is again consistent.
\begin{eqnarray*}
\pi'(x_s : s \in \kappa) &= &\{ [[h_i^{-1}h_j,  k^{-1}_i k_j], x_s^{-1} x_t]\not= 1:  (i,j), (n,m), (s,t ) \in \lambda\}\\
&& \cup \{ x_s \in L : s\in \kappa\}
\end{eqnarray*}

A realization of this type together with the $(h_i : i \in \kappa)$ and $(k_n: n \in \kappa)$ satisfies the demanded properties.

On the other hand, suppose for any cardinal $\kappa$, there exists a saturated extension $\mathcal G$ of $G$ and elements $(h_i : i \in \kappa)$ in $H(\mathcal G)$, $(k_n: n \in \kappa)$ in $K(\mathcal G)$, and $(l_s : s\in \kappa)$ in $L(\mathcal G)$ such that
$$ [[h_i^{-1}h_j,k_n^{-1}k_m],l_s^{-1}l_t] \not = 1\ \ \ \forall  i,j,n,m,s,t\in \kappa,\ i \not= j, n \not= m, s \not=t.$$
%
If $H \ls \t C_G(K/\t C_G(L))$ then one can find $i$ different than $j$ such that $h_i^{-1}h_j>$ is an element of $ \t C_G(K/\t C_G(L))$.
So the index of $C_K(h_i^{-1}h_j/  \t C_G(L))$ in $K$ is bounded. Once more this implies that one can find $n$ different than $m$ such that $k_n^{-1}k_m \in C_G(h_i^{-1}h_j / \t C_G(L))$. 
Thus $[h_i^{-1}h_j,k_n^{-1}k_m]$ is an element of $\t C_G(L)$ or equivalently the index of $C_L([h_i^{-1}h_j,k_n^{-1}k_m])$ has bounded index in $L$. Thus there exists $s$ different than $t$ such that $ [[h_i^{-1}h_j,k_n^{-1}k_m],l_s^{-1}l_t] = 1$ which contradicts our assumption and the Lemma is established.
\qed

Now we are ready to prove the approximate three-subgroup lemma.

\begin{theorem}\label{thm_3sl1}
Let $H$, $K$ and $L$ be three ind-definable subgroups of $G$ which normalize each other.  If
$$H \ls \t C_G(K/\t C_G(L)) \mbox{ and }K \ls \t C_G(L/\t C_G(H)) \mbox{ then } L \ls \t C_G(H/\t C_G(K)).$$
\end{theorem}

\proof
Assume towards a contradiction that $L \not \ls \t C_G(H/\t C_G(K))$. Let $\kappa$ be any cardinal greater or equal to $2^{|T|}$. By the previous lemma we can find $(l_s : s\in (2^\kappa)^+)$ in $L$, $(k_n: n \in (2^\kappa)^+)$ in $K$ and $(h_i : i\in (2^\kappa)^+)$ in $H$ in an enough saturated extension of $G$ such that
$$ [[l_s^{-1}l_t,h_i^{-1}h_j], k_n^{-1}k_m ] \not = 1\ \ \ \forall  i,j,n,m,s,t\in (2^\kappa)^+,\ i \not= j, n \not= m, s \not=t.$$
By the Witt-identity, for every tuple $(i,j,n,m,s,t)$ in $(2^\kappa)^+$ with $i \not= j, n \not= m, s \not=t$ either $[[h_j^{-1}h_i,k_m^{-1}k_n],l_s^{-1}l_t] \not = 1$ or $[[k_n^{-1}k_m,l_t^{-1}l_s], h_j^{-1}h_i] \not = 1$. By Erdos-Rado we can find a subset $I$ of $\kappa^+$ such that for all  $(i,j,n,m,s,t)$ in $I$ with $i \not= j, n \not= m, s \not=t$ the same inequality of the two holds. Without lost of generally, we may assume that $[[h_i^{-1}h_j,k_n^{-1}k_m],l_s^{-1}l_t] \not = 1$ for all  $(i,j,n,m,s,t)$ in $I^6$ with $i \not= j, n \not= m, s \not=t$. But than the previous lemma yields that $H \not\ls \t C_G(K/\t C_G(L))$ which contradicts our assumptions and finishes the proof.
\qed



\subsection{Almost commutator} \

To easer the notation in this subsection we let $\F$ be family of all $A$-definable subgroups of $G$. Note that this family is stable under finite intersection and finite product. Then we may define the almost commutator.
\begin{defn}
For two $A$-ind-definable subgroups $H$ and $K$ of $G$, we define:
$$\[ H, K \]_A :=\bigcap \{L \in \F: 
\, L=L^{N_G(H)}=L^{N_G(K)}, \,  H \lesssim \t C_G(K/L)\}$$
and call it the \emph{almost $A$-commutator} of $H$ and $K$. If $A$ is the empty set we omit the index and just say the almost commutator.
\end{defn}

By Theorem \ref{Thm_sym} the almost commutator is symmetric, i.\ e.\ for two ind-definable subgroups $H$ and $K$ , we have $\[H,K\] = \[K, H\]$. Moreover, it is the intersection of a descending directed system of definable subgroups of $G$.
Note that the ordinary commutator of two $A$-ind-definable groups is not necessary definable nor the intersection of definable subgroups, and hence, other than for the almost centralizer, one cannot compare it with its approximate version.

We will mainly be interested in the almost commutator of normal subgroups of $G$. In this case, the subgroup $\[ H, K \]$ is the intersection of normal subgroups in $G$. Hence we may restrict $\F$ to contain only the normal $A$-definable subgroups of $G$ and the definition simplifies to:
$$\[ H, K \]_A :=\bigcap \{L\in \F: \,  H \lesssim \t C_G(K/L)\}.$$

\begin{defn}
We define the \emph{almost lower $A$-central series} of an $A$-ind-definable subgroup $H$ of $G$:  $$(\t\gamma_1 H)_A = H\ \ \mbox{ and }\ \ (\t\gamma_{i+1} H)_A = \[\t\gamma_i H, H\]_A.$$ Again, if $A$ is the empty set we omit the index.
\end{defn}

\begin{remark}
The almost lower center series is well-defined as $\[H,H\]$ is the intersection of $A$-definable groups and hence $A$-type-definable. So by induction we see that $\tg_{i+1} H = \[\tg_i H,H\]$ is again an $A$-type-definable subgroup.
\end{remark}

In the rest of this subsection, we establish basic properties of the almost commutator in arbitrary groups.

\begin{lemma}\label{Lem_SmlSubF}
For any $A$-ind-definable subgroups $H$ and $K$ of $G$, we have that $H\ls  \t C_G(K/\[H,K\]_A)$. In particular, $\[H,K\]_A$ is the smallest intersection of $A$-definable subgroups for which this holds.
\end{lemma}

\proof
To easer notation we may assume that $A$ is the empty set.  Let $\kappa = 2^{|T|}$. If the index of  $ \t C_H(K/\[H,K\])$ in $H$ is unbounded, up to passing to some extension of $G$, we can find $\{h_i: i \in ({2^\kappa})^+\}$ a set of representatives of cosets of $ \t C_H(K/\[H,K\])$ in $H$. Hence $h_i^{-1}h_j \not\in \t C_H(K/\[H,K\])$, i.\ e.\ the index of $C_K(h_i^{-1}h_j / \[H,K\]))$ in $K$ is unbounded.
Whence, for any $i$ different than $j$, we can find a set of representatives $\{k^{(i,j)}_n:n \in  ({2^\kappa})^+\}$ of cosets of  $C_K(h_i^{-1}h_j / \[H,K\]))$ in $K$. Thus, for $n$ different than $m$, we have that $[h_i^{-1}h_j, {k^{(i,j)}_n}^{-1}k_m^{(i,j)}] \not \in \[H,K\]$. But the group $\[H,K\]$ is the bounded intersection of definable groups $L_\alpha$ with $\alpha$ in some index set $\Omega$. So for any given tuple $(i,j,n,m)$, we find an element $\alpha_{(n,m)}^{(i,j)}$ in $\Omega$ such that  $[h_i^{-1}h_j, {k^{(i,j)}_n}^{-1}k_m^{(i,j)}] \not \in L_{\alpha_{(n,m)}^{(i,j)}}$. Applying Erdos-Rado twice, we can find subsets $I_H$ and $I_K^{(i,j)}$ (for all $(i,j)$ in $I_H\times I_H \setminus \{(i,i): i \in I_H\}$) of $({2^\kappa})^+$ of size at least $\kappa^+$ and $\alpha$ in $\Omega$ such that for all $(i,j)$ in $I_H\times I_H \setminus \{(i,i): i \in I_H\}$ and $(n,m)$ in $I_K^{(i,j)}\times I_K^{(i,j)}\setminus \{(n,n): n \in I_K^{(i,j)}\}$, we have that $\alpha_{(n,m)}^{(i,j)}$ is equal to $\alpha$. But now, these elements witness that $H$ is not almost contained in $\t C_G(K/ L_\alpha)$ which leads to a contradiction and gives the first part of the Lemma. 

Now, let $L$ be an intersection of $A$-definable subgroups such that  $H\ls  \t C_G(K/L)$. Than, this holds for any of the definable subgroups in the intersection. Thus, those subgroups contain $\[H,K\]$ and therefore $L$ contains $\[H,K\]$.
\qed

Using the previous lemma we obtain immediately the following corollaries.

\begin{cor}\label{Cor_EqCenCom} Let $H$ and $K$ be two $A$-ind-definable normal subgroups of $G$ and $L$ be the intersection of normal $A$-definable subgroups of $G$. Then, we have that  $H \ls \t C_G(K/L)$ if and only if $\[H,K\]_A \leq L$.
\end{cor}

\begin{cor}\label{cor_HHleqK}
For any almost commutator of two $A$-ind-definable subgroups $H$ and $K$ and any intersection $L$ of $A$-definable normal subgroups, we have that $\[H,K\]_A \ls L$ if and only if $\[H,K\]_A \leq L$
\end{cor}

\proof
The implication from right to left is trivial. So suppose that $\[H,K\]_A \ls L $. Lemma \ref{Lem_SmlSubF} yields that $H\ls  \t C_G(K/\[H,K\]_A)$. Furthermore, by assumption we have that the intersection of $A$-definable subgroups $\[H,K\]_A \cap L$ has bounded index in $\[H,K\]_A $. So, we have as well that $H\ls  \t C_G(K/(\[H,K\]_A\cap L))$. As $\[H,K\]_A $ is the smallest subgroup for which this holds, we obtain the result.
\qed

The next lemma seems rather trivial but it is essential for almost any proof concerning computations with almost commutators.
\begin{lemma}\label{Lem_ComBas}
For $H$, $K$, $N$ and $M$ be normal $A$-ind-definable subgroups of $G$ we have the following:
\begin{enumerate}
\item If $N \ls H$ and $M \ls K$ then $\[ N, M \]_A \leq \[ H, K \]_A$.
\item The almost commutator of $H$ and $K$ is again normal in $G$ and if $H$ (resp. $K$) is the intersection of definable groups it is contained in $H $ (resp. $ K$).
\item If $H$ and $K$ are assumed to be definable then $N \ls \t C_G(M/H)$ implies that $N \ls \t C_G(M/HK)$
\end{enumerate}
\end{lemma}

\proof
To prove (1), we let $L$ be an arbitrary normal $A$-definable subgroup of $G$ such that $H \ls \t C_G(K/L)$. Since $K \cap M$ is a subgroup of $K$, we have as well that $H \ls \t C_G(K \cap M/L)$. As $N$ is almost contained in $H$, we may replace $H$ by $N$ and obtain $N\ls \t C_G(K \cap M/L)$. Additionally, the almost centralizer of two commensurable $A$-ind-definable subgroup such as $M$ and $K \cap M$ coincides. Thus we conclude that $N\ls \t C_G(M/L)$ or in orther words $\[ N, M\]_A \leq L$. As $L$ was arbitrary, the almost commutator $\[N, M\]_A$ is contained in $\[ H, K \]_A$.

As $H$ and $K$ are normal in $G$, the almost commutator is the intersection of normal subgroups of $G$ and hence normal in $G$. We have trivially that $H \leq \t C_G(K/H)$. So if $H$ is normal and the intersection of definable groups, we conclude that the almost commutator of $H$ and $K$ is contained in $H$.

An inspection of the definition gives immediately (3).
%
%
\qed

\begin{lemma}\label{Rem}
Let $H$ and $K$ be two subgroups of $G$ which are intersection of descending directed systems of $A$-definable sets, i.\ e.\ $H= \bigcap_{i \in I} H_i$ and $K= \bigcap_{s\in S} K_s$ such that for any $i$, $j$ in $I$ and $s$, $t$ in $S$ there exists $n$ in $I$ and $m$ in $S$ such that  $H_i \cap H_j \supseteq H_n$ and $K_s \cap K_t\supseteq K_m$. Then, we have that $$H K = \bigcap_{ (i, s) \in I \times S} H_i K_s.$$
\end{lemma}

\proof
Inclusion from left to right is obvious. So take $c \in  \bigcap_{ (i, s) \in I \times S} H_i K_s$. Thus for all $i \in I$ and $s \in S$ there exists $h_i$ in $H_i$ and $k_s$ in $K_s$ such that $c = h_i k_s$. So the following type over $A$ is consistent.
$$\pi(x,y) = \{x \in H_i : i \in I\} \cup \{y \in K_s : s \in S\} \cup \{c = xy\} $$
Hence by compactness and saturation of $G$, one can find $h \in  \bigcap_{i \in I} H_i=  H$ and $k \in \bigcap_{s\in S} K_s= K$ such that $c =hk$.
\qed

\begin{lemma}\label{Lem_StabMul} Let  $H$, $K$, and $L$ be $A$-ind-definable subgroups of $G$ which normalize each other. Then we have
$$\[HK, L\]_A \leq \[H, L\]_A\cdot \[K,L\]_A.$$
\end{lemma}

\proof To simplify the notation may work in $N_G(H) \cap N_G(K) \cap N_G(L)$ and therefore suppose that $H$, $K$ and $L$  are normal subgroups and restrict the family $\F$ to the normal subgroups of $N_G(H) \cap N_G(K) \cap N_G(L)$ in this proof.
\begin{eqnarray*}
\[H,L\]_A\cdot \[K,L\]_A& = &\bigcap \{ M\in \F : H \ls \t C_G (L/M)\} \cdot \bigcap \{ N\in \F : K \ls \t C_G (L/N)\} \\
& \overset{\ref{Rem}}=& \bigcap \{ M \cdot N :  M,N \in \F ,\, H \ls \t C_G (L/M),\, K \ls \t C_G (L/N)\}
 \end{eqnarray*}
As the product of two groups in $\F$ is again group which belongs to $\F$ and since $ H \ls \t C_G (L/M)$ and $ K \ls \t C_G (L/N)$,  Lemma \ref{Lem_ComBas} (3) yields that $ H \ls \t C_G (L/MN)$ and $ K \ls \t C_G (L/MN)$. So by Lemma \ref{Lem_ProdCon} we obtain $ HK \ls \t C_G (L/MN)$. Thus, the previous set is contained in the following one: 
\begin{eqnarray*}
&\supseteq & \bigcap \{  P\in \F : HK \ls \t C_G (L/P) \} \ \ \ \ \ \ \ \ \ \ \ \ \ \ \ \ \\
&=&  \[HK, L\]_A
 \end{eqnarray*}
This finishes the proof.
\qed

We would like to translate the approximate version of the three subgroup lemma into the notation of almost commutators.  The problem we are facing is that the almost centralizer of a subgroup is not necessarily definable. This leads to our next chapter, where we analyze groups which satisfy a chain condition on centralizer up to finite index. We proof that in those groups the almost centralizer of subgroups are definable and thanks to this result we are able to generalize a theorem of Hall of nilpotent subgroups to almost nilpotent subgroups.

\section{$\M$-groups}\label{sec_Mcgroups}
$\Mn$-groups, i.\ e.\ groups for which every descending chain of centralizers stabilizes after finitely many steps has been of great interest to both group and model theorist (see \cite{AltBag} \cite{Bry} \cite{Bry2} \cite{DerWag}  \cite{PoiWag}). In the field of model theory they appear naturally as definable groups in stable theories. Passing to groups definable in simple theories, we obtain a weaker chain condition, namely any chain of centralizers, each having infinite index in its predecessor, stabilizes after finitely many steps. We want to study group for which any \emph{definable quotient} (i.\ e.\ a definable subgroup $H$ of $G$ quotient by a definable normal subgroup of $H$) satisfies this chain condition.
\begin{defn}
A group $G$ is called \emph{$\M$-group} if for any two definable subgroup $H$ and $N$, such that $N$ is normalized by $H$ there exists natural numbers $n_{HN}$ and $d_{HN}$ such that any sequence of centralizers $$C_{H/N}(g_0  N) \geq \ldots \geq C_{H/N}(g_0N , \dots g_mN) \geq \ldots $$%
each having index at least $d_{HN}$ in its predecessor has length at most $n_{HN}$.
%
\end{defn}

\begin{remark}\label{Rem_defMccom}
Note that any definable subgroup of $G$, any definable quotient of $G$ and any elementary extension is again an $\M$-group.
\end{remark}

For the rest of the section we work in an $\M$-group $G$. One of the crucial property of subgroups of $G$ is that the iterated almost centralizer are definable which we proof below.

\begin{prop}\label{prop_FCdef}
Let $H$ be an arbitrary subgroup of $G$. Then all iterated almost centralizers $\t C_G^n(H)$ are definable.
\end{prop}
\proof
For $n$ equals to $0$ there is nothing to show as the trivial subgroup is always definable. Now assume that $\t C^n_G(H)$ is definable for a given natural number $n$.
%
%
Since $G$ is an $\M$-group and as $\t C^n_G(H)$ is a definable subgroup, there are $g_0, \ldots, g_m \in \t C^{n+1}_G(H)$ and $d \in \omega$ such that for all $h \in \t C^{n+1}_G (H)$:
$$\left[ \bigcap_{i=0}^{i=m} C_G(g_i/ \t C^n_G(H)) : \bigcap_{i=0}^{i=m} C_G(g_i /\t C^n_G(H))  \cap C_G(h / \t C^n_G(H))\right] < d$$
Let $D$ be equal to the definable group $\bigcap_{i=0}^{i=m} C_G(g_i/\t C^n_G(H))$. 
Then the following set is definable.
$$ S := \left\{g \in G : \left[ D:  C_D(g/ \t C^n_G(H))\right] < d \right\}$$
We show that $S = \t C^{n+1}_G(H)$. The inclusion $\t C^{n+1}_G(H) \subset S$ is obvious by choice of the $g_i$'s and $d$. So let $g \in S$. Then we may compute:
\begin{eqnarray*}
[ H : C_H(g/ \t C^n_G(H))] & \leq &\left[H: H\cap D\right] \cdot \left[ H\cap D : C_{H \cap D}(g/ \t C^n_G(H))\right] \\
&\leq &\left[H: H\cap D \right] \cdot \left[ D :  C_D(g/ \t C^n_G(H))\right]< \infty
\end{eqnarray*}
Thus $g$ belongs to $\t C^{n+1}_G(H)$. Hence $\t C^{n+1}_G(H)$ is equal to $S$ and whence definable.
\qed

\begin{remark}\label{rem_FCdef}
Note that all the groups mentioned in the lemma above are stabilized by any automorphism which fixes set-wise $H$. So, if $H$ is an $A$-invariant group, they are indeed definable over $A$. Moreover, for any (type-, ind-) definable (resp. $A$-invariant) subgroup $H$, the iterated almost center of $H$ is (type-, ind-) definable (resp. $A$-invariant).
\end{remark}

 \subsection{Almost commutator and the three-subgroup lemma}\

A consequence of the definability of the almost centralizer in $\M$-groups is that the almost commutator is ``well behaved''. In this subsection, we establish the three subgroup lemma in terms of the almost commutator in $\M$-groups. This enables us to generalize results from \cite{Bry} to our context. 

As  $G$ is an $\M$-group, we have that $H \ls \t C_G(K/\t C_G(L))$ if and only if $\[H, K, L\]$ is trivial. With this equivalence, we may phrase Theorem \ref{thm_3sl1} for $\M$-groups as below:

\begin{cor}\label{Cor_TSL} Let $H$, $K$ and $L$ be three $A$-ind-definable subgroups of $G$ which normalize each other. Then for any $M$ which is an intersection of $A$-definable normal subgroups of $G$, we have that
$$\[H, K, L\] \leq M \mbox{ and }\[K, L, H\] \leq M  \mbox{ imply }\[L, H, K\] \leq M .$$
\end{cor}

\proof
Let $M$ be equal to the intersection of definable subgroups $M_i$ with $i < \kappa$. For any $i$ less than $\kappa$, we may work in the group $G$ modulo $M_i$. which is a quotient of an $\M$-group by a normal definable group and so an $\M$-group as well. Hence, Theorem \ref{thm_3sl1} (working modulo the definable group $M_i$) yields that
$$H \ls \t C_{G}(K/\t C_{G}(L/M_i)) \mbox{ and } K\ls \t C_{G}(L/\t C_{G}(H/M_i))$$
imply
$$L \ls \t C_{G}(H/\t C_{G}(K/M_i)).$$
Which we can translate to
$$\[H, K, L\] \leq M_i \mbox{ and }\[K, L, H\] \leq M_i  \rightarrow \[L, H, K\] \leq M_i $$
So the statement is true for any $M_i$ and hence for the intersection.
\qed

Using symmetry of the almost centralizer, the three-subgroup-lemma and the definablily of the almost centralizer, we may generalize a theorem due to Hall  \cite[Satz III.2.8]{Hup}  for the ordinary centralizer to our context.

\begin{prop}\label{thm_AlCenCompli}
Let $N_0 \geq N_1 \geq \dots \geq N_m \geq \dots$ be a descending sequence of definable subgroups of $G$ normal in $N_0$, and let $H$ be an ind-definable subgroup of $G$. Suppose that for all $i \in \mathbb N$, we have $H \ls \t C_G (N_i/N_{i+1})$. We define for $ i \in \mathbb N$,
$$H_i := \bigcap_{k \in \mathbb N} \t C_H(N_k / N_{k+i}).$$
Then we have that for all positive natural numbers $i$ and $j$, the group $H_i$ is almost contained in $\t C_G(H_j / H_{i+j})$ and $H$ is almost contained in $\t C^i_G(H / \t C_G( N_{j-1} /N_{i+j}))$ and therefore $\[\tg_{i+1} H, N_{j-1} \] \leq N_{i+j}$.

\end{prop}

\begin{remark}
The non-approximate version \cite[Satz III.2.8]{Hup} states that for $H_j$ defined as $\bigcap_{i< \omega} C_H(N_i/N_{i+j})$ one has that $[H_i, H_j] \leq H_{i+j}$ and $[\gamma_{i+1} H, N_{j-1}] \leq N_{i+j}$.
\end{remark}

\proof
As $H$ and the $N_i$'s are ind-definable subgroups of an $\M$-group, we have that each $ \t C_H(N_k / N_{k+i})$ is an ind-definable group. Thus $H_i$ is a bounded intersection of ind-definable groups and thus ind-definable. Moreover, by hypothesis we have that $H_1$ is a bounded intersection of  groups which are commensurable with $H$ and whence it is itself commensurable with $H$. As commensurable groups have the same almost centralizer, it is enough to show the second part of the proposition for the group $H_1$. 

Observe that
$$ H_i \ls \t C_G(H_j / H_{i+j}) =\t C_G(H_j /  \bigcap_{k < \omega} \t C_H(N_{k} / N_{k + i+j}))$$
if and only if for all natural number $k$ we have that
$$ H_i \ls \t C_G(H_j / \t C_G(N_{k} / N_{k+i+j})).$$
So it is enough to show the latter result for arbitary $k \in \mathbb N$. So fix some $k$ in $\mathbb N$ and let $n$ and $m$ be two arbitrary natural numbers. By the definition of $H_n$ we have that  $H_n \leq  \t C_G(N_{k+m} / N_{k+m+n})$. Additionally, using that the almost centralizer satisfies  symmetry modulo definable groups, we obtain that $N_{k+m} \ls \t C_G(H_n / N_{k+m+n})$. So, for any three natural numbers $i$, $j$ and $k$ we have that
\begin{eqnarray}\label{Feq3sl}
H_i & \leq & \t C_G(N_{k} / N_{k+i}) \leq  \t C_G(N_{k} / \t C_G(H_j / N_{k+i+j}))
\end{eqnarray}
as well as
\begin{eqnarray*}
H_j & \leq & \t C_G(N_{k} /\t  C_G(H_i / N_{k+j+i})) =  \t C_G(N_{k} /\t  C_G(H_i / N_{k+i+j})).
\end{eqnarray*}
Applying symmetry to the last equation, we obtain:
\begin{eqnarray}\label{Seq3sl}
N_{k}& \ls & \t C_G( H_j /\t  C_G(H_i / N_{k+i+j})).
\end{eqnarray}
Working in $G/N_{k+i+j}$, we can apply the three subgroup lemma to the equalities (\ref{Feq3sl}) and (\ref{Seq3sl}) since all $N_i$'s and all $H_j$'s normalize each other  and obtain
$$H_j  \ls \t C_G(H_i / \t C_G(N_{k} / N_{k+i+j})).$$
As $k$ was arbitrary, this establishes the first part of the theorem.

In particular, we have that
\begin{eqnarray*}
H_1 & \ls &  \t C_G(H_1 / H_2) \ls \t C_G(H_1 / \t C_G(H_1 / H_3) ) = \t C^2_G(H_1/H_3) \\
& \ls & \dots \ls \t C^i_G(H_1/H_{i+1}) \ls  \t C^i_G(H_1/ \t C_G( N_{j-1} /N_{i+j}))
\end{eqnarray*}
As mention in the beginning of the proof, this implies the same almost inclusion for $H$ which finishes the proof.
\qed

Using the previous result and definability of the almost centralizers, we may find an approximate version of \cite[Lemma 2.4]{Bry} in terms of the almost centralizer.

\begin{cor}
Let $H$ be an ind-definable subgroup of $G$. Then for any $0 < i <j $, we have that
$$ H \ls \t C^i_G\left( H \, \Big/\,  \t C_G\left[ \t C^j_G(H) / \t C^{j-i-1}_G(H) \right] \right)$$
\end{cor}

\proof
For $k < 2j-1$, we let $N_k = \t C_G^{2j-1-k}(H)$ and for $k \leq 2j-1$, we let $N_k$ be the trivial group. As $G$ is an $\M$-group, all $N_k$ are definable. Additionally, we have that $H \ls \t C_G(N_k/ N_{k+1})$. So we may apply Theorem \ref{thm_AlCenCompli} to the ind-definable subgroup $H$ and the sequence of definable groups $N_i$. This gives us that
$$H  \ls \t C^i_G(H / \t C_G( N_{j-1} /N_{i+j})) =\t C^i_G(H / \t C_G( \t C_G^{j}(H)/ \t C_G^{j-i-1}(H)))  $$
\qed

Using the new notion of almost commutator enables us state the previous lemma in this terminology which resembles more to the ordinary result. 
\begin{cor}\label{cor_Comij}
Let $H$ be an $A$-ind-definable group of the $\M$-group $G$. Then for any $0< i < j$, we have that
$$\[(\t\gamma_{i+1}H)_A, \t C^j_G(H)\]_A \leq \t C_G^{j-i-1}(H).$$
\end{cor}

\proof
We have that
$$H \ls \t C^i_G\left(H\big/ \t C_G\left(\t C^j_G(H)/ \t C_G^{j-i-1}(H)\right) \right) = \t C_G\left(H / \t C_G^{i-1}\left(H / \t C_G\left(\t C^j_G(H)/ \t C_G^{j-i-1}(H)\right) \right)\right)$$
Since the iterated almost centralizer is a definable group, this yields that
$$ \[H,H\]_A \leq \t C_G^{i-1}\left(H/ \t C_G\left(\t C^j_G(H)/ \t C_G^{j-i-1}(H)\right) \right)$$
Iterating this process gives us
$$ (\tg_i H)_A \leq \t C_G\left(H/ \t C_G\left(\t C^j_G(H)/ \t C_G^{j-i-1}(H)\right) \right).$$
As the almost centralizer $ \t C_G\left(\t C^j_G(H)/ \t C_G^{j-i-1}(H)\right) $ is definable, this yields the following
$$(\t \gamma_{i+1}H)_A \leq \t C_G\left(\t C^j_G(H)/ \t C_G^{j-i-1}(H)\right).$$
By the same argument, we obtain the final inequation:
$$[(\t\gamma_{i+1}H)_A, \t C^j_G(H)\] \leq \t C_G^{j-i-1}(H).$$
\qed

In the next lemma, we use the approximated three subgroup lemma in terms of the almost commutator to generalize \cite[Lemma 2.5]{Bry} to our context. 

\begin{lemma} Let $H$ and $K$ be two $A$-ind-definable subgroups of $G$ with $K  \trianglelefteq H$ and $l$ be a natural number. If
$$\t C_G((\t \gamma_tK) _A)\sim \t C_G((\t \gamma_tH)_A) \ \ \ t = 1, \dots, l$$
then $ \t C^l_G(K) \sim \t C^l_G(H)$.
\end{lemma}

\proof
We show this by induction on $l$. The case $l$ equals $1$ is trivial. So let's assume that the lemma holds for $l-1$. We need to prove the following intermediate result which is done by induction as well, this time simultaneously on $l$ and $t$:
\begin{claim}
$\[ \t \gamma_{l-t} (H), \t C_G^l(K) \] \leq \t C_G^t(H)$ holds for all $t= 0, \dots , l-1$.
\end{claim}
\proof
Let $l$ be equal to $1$ and therefore $t$ has to be equal to $0$. Then the claim states that $\[H, \t C_G(H)\]$ is trivial which is true by definition. 

Now suppose that the claim holds for all value smaller than $l$. 
Replacing $H$ by $K$, $i$ by $l-1$, and $j$ by $l$ in Corollary \ref{cor_Comij}, we obtain $\[\t\gamma_{l}K, \t C^l_G(K)\] = 1$. This implies that $\t C_G^l(K)$ is almost contained in $\t C_G(\t \gamma_l(K))$ which is by the hypothesis of the lemma is commensurable with $\t C_G(\t \gamma_l(H))$. Thus $\t C_G^l(K) \ls \t C_G(\t \gamma_l(H))$ or in other words $\[\t \gamma_l(H), \t C_G^l(K)\] =1 $. Hence the claim holds for $t$ equals to $0$. \\
Let $t< l$ and assume the claim is true for any value smaller than $t$. Then
$$\[\[\t \gamma_{l-t}H, K\], \t C_G^l(K)\] \leq \[\[\t \gamma_{l-t}H, H\], \t C_G^l(K)\] = \[\t \gamma_{l-t+1}H, \t C_G^l(K)\] \leq \t C^{t-1}_G(H)$$
and
$$\[ \t \gamma_{l-t}H,\[K, \t C_G^l(K)\] \]\leq \[ \t \gamma_{l-t}H, \t C_G^{l-1}(K)\] = \[ \t \gamma_{l-t}H, \t C_G^{l-1}(H)\] \leq \t C_G^{t-1} (H)$$

Thus by Corollary \ref{Cor_TSL} we have
$$\[\[\t \gamma_{l-t}H,\t C_G^l(K) \],K \]  \ls \t C_G^{t-1} (H)$$
By induction hypothesis on $t$, we have that  $\t C_G^{t-1} (H)$ is commensurable with $\t C_G^{t-1} (K) $ and so $ \[\[\t \gamma_{l-t}H,\t C_G^l(K) \],K \]$ is almost contained in $\t C_G^{t-1} (K)$.
As $\t C_G^{t-1} (K)$ is $A$-definable, using Corollary \ref{cor_HHleqK}, we obtain that
$$\[\[\t \gamma_{l-t}H,\t C_G^l(K) \],K \]  \leq \t C_G^{t-1} (K)$$
Thus $\[\t \gamma_{l-t}H,\t C_G^l(K) \]$ is almost contained in $ \t C_G^{t} (K)$ which is commensurable with $ \t C_G^{t} (H)$ by the initial induction hypothesis. Again by Corollary \ref{cor_HHleqK} almost contained can be replaced by contained, which gives us that
$$\[\t \gamma_{l-t}H,\t C_G^l(K) \] \leq \t C_G^{t} (H)$$
which finishes the proof of the claim.
\qed$_{(\mbox{claim})}$

Now taking $t$ equals to $l-1$, we obtain $ \[ H, \t C_G^l(K)\] \leq \t C_G^{l-1}(H)$ which implies that $\t C_G^l(K)$ is almost contained in $\t C_G^l(H)$. On the other hand, we have $ \[K, \t C_G^l(H)\] \leq \t C_G^{l-1}(H)$ which by induction hypothesis is commensurable with $\t C_G^{l-1}(K)$. This implies that  $\t C_G^l(H)$ is almost contained in $\t C_G^l(K)$. Combining these two results, we obtain that $\t C_G^l(K)$ is commensurable with $\t C_G^l(H)$ which finishes the proof. \qed

\subsection{Passing to definable groups with the same properties}\label{sec_env} \

We want to proof that subgroups of any $\M$-group which are abelian, (almost) nilpotent or (almost) solvable are almost contained in subgroups which admit the same algebraic property. For this, the crucial property of subgroups of $\M$-groups is that the iterated almost centralizer are definable. Recall that we fixed an $\M$-group $G$.

The following results are needed to analyze envelops of subgroups of such groups. Some of them can be found in \cite{cm_defenv} for the FC-centralizer of an arbitrary subgroups.
\begin{lemma}\label{lem_norFC} \cite[Lemma 2.1]{cm_defenv}
Let $H$ an $A$-invariant subgroup of $G$, $N$ an $A$-invariant subgroup of $H$ and $n \in \omega$. Then the following holds:
\begin{enumerate}
\item $H$ normalizes $\t C_G(H)$.
\item If $H$ is normal in $G$ ($N \leq H \trianglelefteq G$) and $g \in G$. Then $ C_G(g/N) \leq C_G(g/H)$.
\item For any $h \in H$, we have $C_H(h/ \t C_n(G))=C_H(h/ (\t C_n(G) \cap H))$
\item $H \cap \t C_n(G)$ is a subgroup of $\t C_n(H)$.
\end{enumerate}
\end{lemma}

\proof

\begin{enumerate}
\item Let $h \in H$. Since conjugation by $h^{-1}$ is an automorphism, we can compute $[H : C_H(g^h)] = [ H^{h^{-1}}: (C_H(g)^h)^{h^{-1}}] = [H: C_H(g)]$. Hence $[H : C_H(g^h)]$ is bounded if and only if $[H: C_H(g)]$ is bounded. So $g^h \in \t C_G(H)$ for any $g \in \t C_G(H)$ and $h \in H$. Thus $H$ normalizes $\t C_G(H)$.
\item Fix arbitrary elements $a \in  C_G(g/N)$ and $h \in H$. We have to find $h' \in H$ such that $a \cdot gh = gh' \cdot a$. As $a \in  C_G(g/N)$, there exists $k \in N$ such that $a\cdot g= gk \cdot a$. Hence $a \cdot gh = gk \cdot a \cdot h = gk \cdot aha^{-1} \cdot a $. Note that $aha^{-1} \in H$ as $H$ is normal in $G$. Thus, as $N$ is a subgroup of $H$, $h' = k \cdot ah a^{-1}$ is an element of $H$ and we can conclude.
\item We show the two inclusions. So, let $k$ be an element of $C_H(h/ \t C_n(G))$. For any $ f \in  \t C_n(G) \cap H$, there exists some $f' \in  \t C_n(G)$ such that $k\cdot h f = hf' \cdot k$. Furthermore, as $k \cdot hf$ belongs to $H$, the element $hf' \cdot k $ does as well. Since $h$ and $k$ are elements of $H$, we can deduce that $f' \in H$. Hence $k \in C_H(h/ (\t C_n(G) \cap H))$. On the other hand, let $k$ be an element of $C_H(h /(\t C_n(G) \cap H))$ and $f \in \t C_n(G)$. Since $k \in C_H(h/ (\t C_n(G) \cap H))$, there exists some $e \in \t C_n(G) \cap H$ such that $ k \cdot  h  = h e \cdot  k$. We can compute:

$$ k \cdot  h f = h e \cdot  k \cdot f = h( e\cdot  k f k^{-1})\cdot  k  $$
\
Now $e\cdot  k f k^{-1}$ is an element of $\t C_n(G)$ as it  is normalized by $H$. As $f$ was arbitrary, we may conclude that $k$ is an element of  $C_H(h /\t C_n(G))$.
\item We show this by induction on $n$. So if $h \in H \cap \t Z(G)$, then $[G : C_G(h)]$ is bounded. As $[H : C_H(h)] = [G \cap H : C_G (h) \cap H] \leq [G: C_G(h)]$, $h$ belongs also to $\t Z(H)$. Now, suppose it is true for $n$. Let $h \in H \cap \t C_{n+1}(G)$. Hence, the index $[G: C_G(h/ \t C_n(G))]$ is bounded and as before $[H: C_H(h/ \t C_n(G))]$ is bounded as well. As $h$ is an element of $H$ and by (3) of this lemma, the group $C_H(h/ \t C_n(G))$ is the same as $C_H(h/ (\t C_n(G) \cap H))$. By induction hypothesis, $\t C_n(G) \cap H\leq \t C_n(H)$. Note that all considered groups are normal subgroups of $H$. Hence by (2) we have $C_H(h/ (\t C_n(G) \cap H))\leq C_H(h/ \t C_n(H)) $. Finally, we can conclude that the index $[H : C_H(h/ \t C_n(H))]$ is bounded. So, $h \in \t C_{n+1}(H)$.\qed
\end{enumerate}

\subsubsection{Abelian groups}

\begin{defn}
A group $H$ is called \emph{finite-by-abelian} if there exists a finite normal subgroup $F$ of $H$ such that $H/F$ is abelian.
\end{defn}

\begin{defn}
A group $H$ is an \emph{almost abelian group} if the centralizer of any of its element has finite index in $H$. If the index of these elements can be bounded by some natural number we call it an \emph{bounded almost abelian group}.
\end{defn}

\begin{remark}
If we consider a definable almost abelian subgroup of an $\aleph_0$-saturated group, we can always bound the index of the centralizers by some natural number $n$ by compactness. Hence, any definable almost abelian subgroups is a \emph{bounded almost abelian group}. Additionally, note that the almost center of any group is always an almost abelian group.
\end{remark}

We will need the following classical group theoretical result which is a theorem of Bernhard H. Neumann.

\begin{fact}\label{fac_neum}  \cite[Theorem 3.1]{bhn}.
Let $H$ be a bounded almost abelian group. Then its derived group is finite. In particular, $H$ is finite-by-abelian.
\end{fact}
Now, we are ready to investigate the abelian case. The proof follows the one of the corresponding theorem for simple theories in \cite{cm_wsg}.
\begin{prop}
Every abelian subgroup of an $\M$-group is contained in a definable finite-by-abelian subgroup.
\end{prop}

\begin{proof} Let $H$ be an abelian subgroup of the $\M$ group $G$. We consider the family of centralizers of elements of $H$ in the ambient group $G$. As $G$ is an $\M$-group there are elements $h_0, \dots , h_{n-1}$ in $H$ and a natural number $d$ such that for every element $h$ in $ H$, $[C : C \cap C_G(h)] \leq d$ for $C := \bigcap_{i=0}^{n-1} C_G(h_i)$. Note first that $H$ is contained in $C$ as $H$ is abelian.  Thus by the choice of the $h_i$'s we can deduce that $H$ is a subgroup of $\t Z (C)$. Observe that $\t Z ( C)$ is a definable (Proposition \ref{prop_FCdef}) almost abelian group and so its derived subgroup is finite (Fact \ref{fac_neum}). Hence $\t Z(C)$ is a finite-by-abelian subgroup of $G$ which contains $H$.
\end{proof}

\subsubsection{Solvable groups}
\begin{defn}
A group $H$ is \emph{almost solvable} if there exists a normal \emph{almost series} of finite length, i.\,e.\ a finite sequence of normal subgroups $H_0, \dots, H_n$ of $H$ such that
$$H = H_0 \trianglerighteq H_1  \trianglerighteq \dots  \trianglerighteq H_n = \{1\}$$
and such that $H_i /H_{i+1}$ is an almost abelian group for all $i < n$. The least such natural number $n \in \omega$ is called the \emph{almost solvable class} of $G$.
\end{defn}

\begin{defn}
Let $H$ be a group and $S$ be a definable almost solvable subgroup of class $n$. We say that $S$ \emph{admits a definable almost series} if there exists a family of definable normal subgroups $\{S_i: i \leq n\}$ of $S$ such that $S_0$ is the trivial group, $S_n$ is equal to $S$ and $S_{i+1} / S_i$ is almost abelian.
\end{defn}

We recall the following definition.

\begin{defn}
A family $\F$ of subgroups is \emph{uniformly commensurable} if there exists a natural number $d$ such that for each pair of groups $H$ and $K$ from $\F$ the index of their intersection is smaller than $d$ in both $H$ and $K$.
\end{defn}

We proof first that any almost solvable subgroup of $G$ is contained in a definable almost solvable subgroup of the same class and deduce then the result for solvable groups. For this we need the next theorem which is originally due to Schlichting. A proof and the following definable version can be found in \cite[Theorem 4.2.4]{fow_simth}.

\begin{theorem}[Schlichting's Theorem]\label{thm_sch}
Let $H$ be a group and $\mH$ be a family of definable uniformly commensurable subgroups. Then there exists a definable subgroup $N$ of $H$ commensurable with all members of $\mH$ and invariant under all automorphisms of $H$ which stabilize the family $\mH$ setwise. We have $\bigcap_{H \in \mH} H \subset N \subset \langle \mH \rangle$. In fact, $N$ is a finite extension of a finite intersection of elements in $\mH$. In particular, if $\mH$ consists of relatively definable subgroups, then $N$ is relatively definable.
\end{theorem}

All proofs in the subsection are analog to the onces given in \cite{cm_defenv} for the corresponding results for simple theories.

\begin{prop}\label{thm_FCsol}
Let $S$ be an almost solvable subgroup of $G$ of class $n$. Then there exists a definable almost solvable group of class $n$ which admits a definable almost series.
\end{prop}

\begin{proof}
Choose $S \trianglerighteq S_1 \trianglerighteq \ldots \trianglerighteq S_n = \{1\}$ an almost series for $S$. We construct inductively an ascending chain of definable subgroups $Z_0 \trianglelefteq Z_1 \trianglelefteq \ldots \trianglelefteq Z_n$ such that:
\begin{enumerate}
\item $Z_0$ is the trivial group;
\item $Z_i$ is normal in $Z_j$ for all $i \leq j \leq n$;
\item $Z_i$ contains $S_{n-i}$ for every $i \leq n$;
\item $S$ normalizes $Z_i$ for every $i \leq n$;
\item $Z_i / Z_{i-1}$ is an almost abelian group for every $i \leq n$.
\end{enumerate}

We construct the $Z_i$'s recursively. So let $Z_0$ be equal to the trivial group. Now, fix a natural number $k$ greater than $0$ and assume that $Z_0, \ldots, Z_{k-1}$ are constructed and have the desired properties. For $i <k$, let $N_i$ be the normalizer of $Z_i$ which is definable as $Z_i$ is definable. By (2) and (4), for every $j,i < k$ we have that $N_i$ contains $S$ and $Z_j$, hence these are a subgroups of $N := \bigcap_{i < k} N_i$ which is a definable subgroup of an $\M$-group and thus $\M$ as well. So we may work in $N$ from now on, which ensures that all $Z_i$'s which have been chosen are normal in the ambient group. We consider the following family of subgroups of $N$:
$$ \G_{k} = \left\{ C_N(b/ Z_{k-1}) : b \in S_{n-k}\right\}$$
Since $N$ is an $\M$-group, there are $b_0, \ldots, b_l \in  S_{n-k}$ and $d \in \mathbb N$ such that
$$\left[\bigcap_{i=0}^{l} C_N(b_i/ Z_{k-1}) : \bigcap_{i=0}^{l} C_N(b_i/ Z_{k-1} \cap C_G(b/ Z_{k-1})\right] < d$$
for all $ b \in S_{n-k}$. Let
$$C_k := \bigcap_{i=0}^{l} C_N(b_i/ Z_{k-1}).$$
%
%
%
Note that as $C_N(b/Z_{k-1})$ contains $Z_{k-1}$ for all $b \in S_{n-k}$, the group $C_k$ contains $Z_{k-1}$ as well.

By (3), the group $Z_{k-1}$ contains $S_{n-k+1}$ and the quotient $S_{n-k} / S_{n-k+1}$ is an almost abelian group. Hence $Z_{k-1} S_{n-k} / Z_{k-1}$ is an almost abelian group as well. Therefore, as $C_k$ is the finite intersection of centralizer of elements of $S_{n-k}$ in $N$ modulo $Z_{k-1}$, we have that the index $$[Z_{k-1} S_{n-k} / Z_{k-1} :  (Z_{k-1} S_{n-k} \cap C_k )/ Z_{k-1}]$$ is finite.

Since $S_{n-k}$ is a normal subgroup of $S$, for any element $s$ in $S$ the conjugate $C_k^s$ of $C_k$ is a finite intersection of elements in $\G_k$. Thus, the groups $C_k$ and $C_k^s$ are commensurable and their index is bounded by $md$.
So $\mH_k = \{ C_k^s : s \in S\}$ is a family of definable uniformly commensurable subgroups of $G$. So Schlichtings's theorem yields a definable subgroup $D_k$ of $N$ normalized by $S$ and commensurable with $C_k$. Note that Schlichting's theorem also insures that $D_k$ contains $Z_{k-1}$. As $D_k$ and $C_k$ are commensurable and $(Z_{k-1} S_{n-k} \cap C_k )/ Z_{k-1}$ has finite index in $Z_{k-1} S_{n-k} / Z_{k-1}$, the group $(Z_{k-1} S_{n-k} \cap D_k) / Z_{k-1}$ has finite index in $Z_{k-1} S_{n-k} / Z_{k-1}$ as well. We let $B_k$ be the group $D_k S_{n-k} / Z_{k-1}$, which is a finite union of cosets of $D_k / Z_{k-1}$ and therefore definable. Moreover, $B_k$ is commensurable with $C_k/ Z_{k-1}$. Let $Z_k$ be equal to $\t C_{D_k S_{n-k}} (B_k)$ and observe that it is as well definable.

\begin{claim}
\begin{enumerate}
\item $Z_i$ is normal in $Z_k$ for all $i \leq k$;
\item $Z_k$ contains $S_{n-k}$;
\item $S$ normalizes $Z_k$;
\item $Z_k / Z_{k-1}$ is an almost abelian group.
\end{enumerate}
\end{claim}

\proof
\begin{enumerate}
\item As $Z_i$ is a normal subgroup of $N$ which contains $Z_k$, we obtain the result.
\item Let $g \in  S_{n-k}\setminus Z_k$. Then $[D_{k} S_{n-k} : C_{D_{k} S_{n-k}}(g /Z_{k-1})]$ is infinite. Again, as $C_k$ is commensurable with $D_k$ and $[D_{k} S_{n-k} : D_k]$ is finite, the index $[C_k : C_{C_k}(g/ Z_{k-1})]$ has to be infinite as well. Hence, by choice of $C_k$, $g \not\in S_{n-k}$ which yields a contradiction.
\item As $S$ normalizes $D_k$, $S_{n-k}$ and $Z_{k-1}$, it normalizes the group $B_k$ as well as its almost centralizer $\t C_{D_k S_{n-k}} (B_k)$ in $D_k S_{n-k}$ .
\item This is clear by the definition of $Z_k$.
\end{enumerate}
\qed$_{claim}$

This finishes the construction of the $Z_i$'s. So $Z_n$ is a definable almost solvable group with a definable almost series of the same length as $S$ which contains $S$.
\end{proof}




\begin{lemma}\label{Lem_SolSub2n}
Any definable group in an $\aleph_0$-saturated model which is almost solvable of class $n$ and admits a definable almost series has a definable subgroup of finite index which is solvable of class at most $2n$. 
\end{lemma}

\begin{proof}  Let $H$ be a definable almost solvable group of class $n$ and suppose that  $H = H_0 \trianglerighteq H_1 \trianglerighteq \ldots \trianglerighteq H_n = \{1\}$ is a definable almost series for $H$, i.\ e.\  every $H_i$ is definable and $H_i/H_{i+1}$ is an almost abelian group. By compactness, these are bounded almost abelian group.
Using Fact \ref{fac_neum} we deduce that the quotient group $[H_i, H_i] / H_{i+1}$ is finite. Moreover, as all $H_i$'s are normal subgroups of $H$, the group $[H_i, H_i] / H_{i+1}$ is normalized by $H$. Hence, for any $h$ in $[H_i,H_i]$ the quotient $[h, H]/  H_{i+1}$ is finite, i.\ e.\ the index of $C_H(h/H_{i+1})$ in $H$ is finite. Hence, the definable group $C_H([H_i, H_i] / H_{i+1})$ is the finite intersection of centralizers which have finite index in $H$ and  whence it has finite index in $H$ as well. Omitting the parameters needed to define  $ C_H ([H_i, H_i] / H_{i+1})$, we conclude that it contains the connected component $H^0$ of $H$. This implies that
$$ [[H_i, H_i ], H^0] \leq H_{i+1}.$$

Now, we show by induction on $k$ that 
$$(H^0)^{(2k)}\leq H_k.$$
Let $k$ be equal to $0$. We obtain that $$(H^0)^{(2)} =  [[H^0, H^0], [H^0, H^0]]  \leq  [[H_0, H_0], H^0] \leq H_{1}.$$ Suppose the statement is true for $k$. Then we compute:
$$(H^0)^{(2k+2)} =[[(H^0)^{(2k)}, (H^0)^{(2k)}], [(H^0)^{(2k)}, (H^0)^{(2k)}]] \leq [[H_k, H_k], H^0] \leq H_{k+1}  $$
This finishes the induction.

Hence $(H^0)^{(2n)}$ is a subgroup of the trivial group $H_n$, whence it is trivial as well and therefore $H^0$ is solvable of class at most $2n$. This can be expressed by a formula. So it is implied by finitely many of the formulas defining $H^0$. As $H^0$ is the intersection of a directed system of definable subgroups, this also has to be true in one of those groups. Thus, one can find a definable solvable group of class $2n$ which has finite index in $H$.
\end{proof}


Now, we are able to prove the solvable case.
\begin{theorem}
Let $G$ be an $\M$-group and $H$ be an almost solvable subgroup of class $n$. Then there exists a definable solvable group $D$ of class at most $2n$ such that $H \cap D$ has finite index in $H$.
\end{theorem}

\begin{proof}
Theorem \ref{thm_FCsol} applied to $H$ gives us a definable almost solvable group $K$ of class $n$ containing $H$ which admits a definable almost series. By Lemma \ref{Lem_SolSub2n}, the group $K$ has a definable subgroup $D$ of finite index which is solvable of class at most $2n$.
\end{proof}

\subsubsection{Nilpotent groups} \

\begin{defn}
A group $H$ is \emph{almost nilpotent} if there exists an \emph{almost central series} of finite length, i.\,e.\ a sequence of normal subgroups of $H$
$$\{1\} \leq H_0 \leq H_1 \leq \dots \leq H_n =H$$
such that $H_{i+1} / H_i$ is a subgroup of $\t Z(H/H_i)$ for every $i \in \{0, \dots, n-1\}$.
We call the least such $n \in \omega$, the \emph{almost nilpotency class} of $H$.
\end{defn}
\begin{remark}
The iterated almost centers of any almost nilpotent group $H$ of class $n$ form an \emph{almost central series} of length $n$.
\end{remark}

In this section we prove that any almost nilpotent subgroup of class $n$ is almost contained in a definable nilpotent group of class at most $2n$. To do so, we need the following fact.

\begin{fact}\cite[10]{RB} \label{fact_daniel}
Let $H$ be a group an $K$ and $N$ be two subgroups of $H$ such that $N$ is normalized by $K$. If the set of commutator
$$\{ [k, n] : k \in K, n\in N\}$$
is finite, then the group $[K,N]$ is finite.
\end{fact}

Note that in the following proposition the ambient group $L$ is not assumed to be an $\M$-group. It is a generalization of a result by Neumann (Fact \ref{fac_neum}) stating that the derived group of a bounded almost abelian group is finite.

\begin{prop}\label{prop_HKfin} Let $L$ be a group and let $H$ and $K$ be two definable subgroups of $L$ such that $H$ is normalized by $K$. Suppose that 
$$K \leq \t C_L(H) \ \ \mbox{and} \ \ H \leq \t C_L(K).$$ 
Then we have that $[K,H]$ is finite.
\end{prop}

\proof
Assume that $L$ is at least $\aleph_2$-saturated. As $K$ is definable and contained in the almost centralizer of $H$ in $L$ there exists a minimal bound for the cardinality of the conjugacy of elements of $K$ in $H$. Let $d$ be this minimal bound and fix some element $k$ of $K$ for which the conjugacy class of $k$ in $H$ has size $d$. Choose $1, h_2, \dots, h_d$ be a set of right coset representatives of $H$ modulo $C_L(k)$. Thus
$$ k_1=k,\ \ k_2 = k^{h_2},\ \ \dots,\ \ k_d = k^{h_d}$$
are the $d$ distinct conjugates of $k$ in $ H$. We let $C$ be equal to the centralizer $C_{K}( h_2, \dots, h_d)$. As $H$ is contained in $\t C_L(K)$, we have that the group $C$ has finite index in $K$. Choose some representatives $a_1, \dots, a_n$ of right cosets of $K$ modulo $C$. Note that their conjugacy classes by $H$ are finite by assumption. Thus the group $F$ defined as $\langle k^{H}, a_1^{H}, \dots, a_n^{H} \rangle$ is finitely generated and additionally normal in $H$. Also we have that $K$ is equal to $CF$.

Now, we want to prove that the group $F$ contains the definable set
$$D:= \{ [g, h] : g \in K, h \in H\}.$$
So let $g$ be an arbitrary element of $K$ and $h$ of $H$. Choose $c$ in $C$, $f$ in $F$, such that $g=cf$.
We have that
$$[g, h] = [cf, h] =[c, h]^f[f, h]$$
As $F$ is normalized by $H$, 
we have that $[f, h]$ belongs to $F$. It remains to show that $[c, h]$ does as well.

Let $w= ck$. 
As $c$ commutes with $h_2, \dots, h_d$
the conjugates
$$ w=c k \ \ w^{h_2},= c k_2  \ \ \dots,\ \  w^{h_d}= c k_d $$
are all different. As $d$ was chosen to be maximal, these have to be all conjugates of $w$. So there are $i$ and $j$ less or equal than $d$, 
such that
$$h^{-1} w h = c k_i \ \mbox{ and } \ h^{-1} kh =  k_j$$
and we have that
$$[c, h] = c^{-1} h^{-1} ch= c^{-1} (h^{-1}c kh) (h^{-1} k ^{-1}h) =  c^{-1} c k_i k_j^{-1} = k_i k_j^{-1}.$$
As all $k_i$'s belong to $F$, the commutator $[c, h] $ does as well and thus we conclude that $D$ is contained in $F$. Since $D$ is definable and $L$ is $\aleph_2$-saturated, it is either finite or uncountable. But the group $F$ is finitely generated and therefore countable. Hence $D$ has to be finite and whence  $[K, H]$ is finite by Fact \ref{fact_daniel}.
\qed

\begin{prop}\label{Prop_ComFin}
We have that the commutator $[\t Z(G), \t C_G(\t Z(G))]$ is finite.
\end{prop}

\proof
As $G$ is an $\M$-group, the normal subgroups $\t Z(G)$ and $\t C_G(\t Z(G))$ are definable. As trivially $\t C_G(\t Z(G))$ is contained in itself and $$\t Z(G) = \t C_G(G) \leq \t C_G(\t C_G(\t Z(G))),$$ we may apply Proposition \ref{prop_HKfin} to these two subgroups and obtain the result.
\qed

\begin{prop} \label{Lem_DefEnvNil}
Let $G$ be an $\M$-group and let $H$ an almost nilpotent subgroup of $G$ of class $n$. Then there exists a definable nilpotent subgroup $N$ of $G$ of class at most $2n$  which is normalized by $N_G(H)$ and almost contains $H$.
%
%
\end{prop}

\proof
Suppose the nilpotency class of $H$ is $n$.
We construct inductively on $i\leq n$ the following subgroups of $G$: In the $i$th step we find  a definable subgroup $G_i$ of $G$ and two definable normal subgroups $ N_{2i-1} $ and $ N_{2i}$ of $G_i$ all normalized by $N_G(H)$ such that:
\begin{itemize}
\item $H \ls G_i$;
\item $ \t Z_i (H) \cap G_i \leq N_{2i}$;
\item $[N_{2i-1},G_i ] \leq N_{2(i-1)}$;
\item $[N_{2i}, G_i] \leq N_{2i-1}$;
\item $G_i \leq G_{i-1}$.
\end{itemize}

Once the construction is done, letting $N$ be equal to $N_{2n}$ gives a definable nilpotent subgroup normalized by $N_G(H)$ and of class at most $2n$  which is witnessed by the sequence $$N_0 \cap G_n \leq N_1 \cap G_n \leq \dots \leq N_{2n} \cap G_n.$$


So, let $N_0$ be the trivial group and $G_0$ be equal to $G$.

Now, assume that $i>0$ and that for $j < i$ and $k  < 2j -1$ the groups $N_k$ and $G_{j}$ have been constructed. We work in the  quotient $\mathbb G = G_{i-1} / N_{2(i-1)}$ which is an $\M$-group and we let $\mathbb H = (H\cap G_{i-1}) / N_{2(i-1)}$ which is obviously normalized by $N_G(H)$. The first step is to reduce $\mathbb G$ such that $\t C_\bG(\bH) = \t Z(\bG)$. Observe that the preimage of $\t C_{\bG}(\bH)$ in $G_{i-1}$ contains $\t Z_i(H) \cap G_{i-1}$ as  $ \t Z_{i-1} (H) \cap G_{i-1}$ is contained in $N_{2(i-1)}$. 

If there is $g_0\cdot N_{2(i-1)} \in \t C_\bG(\bH) \setminus \t Z(\bG)$, we consider the family
$$ \mH = \{ C_\bG (g_0^h \cdot N_{2(i-1)}) : h \in N_G(H) \}$$
Note that as $\mathbb H$ is normalized by $N_G(H)$ all members of $\mH$ almost contain $\mathbb H$. Moreover, as $\bG$ is an $\M$-group there exists a finite intersection $\b F$ of groups in $\mH$ such that any bigger intersection has index at most $d$. Thus the family 
$$\{ \b F^h: h \in N_G(H) \}$$ 
is  uniformly commensurable. So by Theorem \ref{thm_sch} there is a definable subgroup $\b C_0$ of $\bG$ which is invariant under all automorphisms which stabilizes the family setwise, thus normalized by $N_G(H)$, and commensurable with $\b F$. Moreover $\b F \cap \b H$ is commensurable with  $C_\bH (g_0 \cdot N_{2(i-1)})$ which has finite index in $\b H $ as $g_0 \cdot N_{2(i-1)}$ belongs to $\t C_\bG(\bH)$. Over all we obtain that  
$$\b C_0 \cap \bH \sim \b H .$$ 
If now, there is $g_1 \in \t C_{\b C_0} (\bH \cap \b C_0)\setminus \t Z(\b C_0)$, we can redo the same construction and obtain a $\b C_1$. 
This process has to stop after finitely many steps, as for every $j$ the index of $C_\bG(g_0  \cdot N_{2(i-1)}, \dots , g_{j+1} \cdot N_{2(i-1)})$ in $C_\bG(g_0 \cdot N_{2(i-1)}, \dots , g_{j} \cdot N_{2(i-1)})$ is infinite by construction. Letting $\b C$ be equal to $\bigcap_i \b C_i$, we found a definable subgroup of $\bG$ (thus an $\M$-group), such that $\t C_{\b C} (\bH) = \t Z(\b C)$, which is normalized by $N_G(H)$ and whose intersection with $\bH$ has finite index in $\bH$.

The next step is to define $G_i$, $N_{2i-1}$ and $N_{2i}$.
As $\b C$ is an $\M$-group, Proposition \ref{Prop_ComFin} yields that the commutator $\b Z= [\t Z(\b C), \t C_{\b C}(\t Z(\b C))]$ is finite. Since $\t Z(\b C)$ and $\t C_{\b C}(\t Z(\b C))$ are characteristic subgroups of $\b C$, we have that $\b Z$ is normalized by $N_G(H)$ and contained in $\t Z(\b C)$. Note additionally that the group $\t C_{\b C}(\t Z(\b C))$ has finite index in $\b C$ by symmetry. Thus  $\b G_i =  \t C_{\b C}(\t Z(\b C)) \cap  C_{\b C} (\b Z)$ has finite index in $\b C$. We let $\b N_1 =   \b Z \cap \b G_i$, a finite subgroup of the center of $\bG_i$, and $\b N_2 =   \t Z(\b C) \cap \b G_i =  \t Z(\bG_i)$, which is contained in $Z(\bG_i/ \b N_1)$. Note that all groups to define $\b G_i$, $\b N_1$ and $\b N_2$ are all characteristic subgroups of $\b C$ and thus $\b G_i$, $\b N_1$ and $\b N_2$ are normalized $N_G(H)$. Moreover, $\b N_1$ and $\b N_2$ are normal subgroups of $\b G$.
Let $G_i$, $N_{2i-1}$ and $N_{2i}$  be the preimage of $\b G_i$, $\b N_1$ and $\b N_2$ in $G$  respectively which satisfy all demands. This finishes the construction and therefore the proof.
\qed

\begin{cor}\label{Cor_NorNilEnv}
If $H$ is a normal nilpotent subgroup of $G$ of class $n$, there is a definable normal nilpotent subgroup of $G$ that contains $H$ of class at most $3n$.
\end{cor}

\proof
By the previous proposition, we can find a definable normal nilpotent subgroup $N$ of $G$ of class at most $2n$ that almost contains $H$. Thus, the group $HN$ is a finite union of coset of the definable subgroup $N$ in $G$. Therefore, we have that $HN$ is a definable normal nilpotent subgroup of class at most $3n$ which contains $H$.
\qed


\section{Fitting subgroup of $\M$-groups}\label{Sec_Fit}

The goal of this subsection is to prove that the Fitting subgroup of any $\M$-group is nilpotent. 
The first step is to show that any locally nilpotent subgroup, thus in particular the Fitting subgroup, of an $\M$-group is solvable.

\begin{prop}
Any locally nilpotent subgroup of an $\M$-group is solvable.
\end{prop}

The proof is inspired by the corresponding result for type-definable groups in simple theories \cite[Lemma 3.6]{PalWag}. For sake of completeness we give a detailed proof.
\proof
Let $K$ be a locally nilpotent subgroup of an $\M$-group $G$. Let $m$ be the minimal natural number such that each descending chain of centralizer in $G$ of infinite index has length at most $m$. We consider all sequences of the form
$$G =C_G(g_1) > \dots > C_G(g_1, \dots g_{m})$$
such that every centralizer has infinite index in its predecessor and let $\S$ be the collection of such tuples $\bar g = (g_1, \dots, g_{m})$. Note that the first element of the tuple is always an element of the center of $G$. We prove by induction that $  C_K(g_1, \dots g_{m-i})$ is solvable for any $\bar g$ in $\S$ and all $i \leq m$ by induction on $i$.

For $i=0$, the group $C_G(g_1, \dots g_{m})$ is a definable almost abelian group. Using Fact \ref{fac_neum} we obtain that its derived group is finite. So the derived group of $ C_K(g_1, \dots g_{m})$ is a finite subgroup of the locally nilpotent group $K$, hence it is nilpotent and whence $C_K(g_1, \dots g_{m})$ is solvable.

Now we assume that the induction hypothesis holds for $i$ greater or equal to $0$ and any $\bar g$ in $\S$. We consider the group $ C_K (g_1, \dots g_{n-i-1})$. By the induction hypothesis, we know that for any $g$ in $G$ for which $C_G(g)$ has infinite index in $C_G(g_1, \dots g_{n-i-1})$, the group $C_K(g_1, \dots g_{n-i-1},g)$ is solvable. Therefore, letting $H$ be equal to the locally nilpotent group $C_K (g_1, \dots g_{n-i-1})$ and replacing $G$ by $C_G(g_1, \dots g_{n-i-1})$ (which is still an $\M$-group as it is a definable subgroup of an $\M$-group) yields that for any $g$ such that $C_G(g)$ has infinite index in $G$, the centralizer $C_H(g)$ is solvable.

Now, we fix some natural numbers $n$ and $d$ such that each descending chain of centralizer in $G$ modulo $\t Z(G)$ of index greater than $d$ has length at most $n$.

If $H$ is contained in the definable almost abelian group $\t Z(G)$, the same argument as above shows that $H$ is solvable. Thus, we may suppose that $H$ is not contained in the almost center of $G$. As $H$ is locally nilpotent, we can find a nilpotent subgroup $H_0$ of $H$ for which this holds, i.\ e.\ the group $H_0 /\t Z(G)$ is non-trivial. As $H_0$ is nilpotent, the subgroup $C_{H_0}(H_0 / \t Z(G))$ strictly contains $\t Z(G) \cap H_0$.
Take an element $h_0$ in their difference. If $C_H(h_0/\t Z(G))$ has index greater than $d$ in $H$, one can find a nilpotent subgroup $H_1$ of $H$ which contains $H_0$ such that $C_{H_1}(h_0/ \t Z(G))$ has index greater than $d$ in $H_1$ as well.
Choose again an element $h_1$ in $C_{H_1}(H_1 /\t Z(G)) \backslash \t Z(G)$, so $C_H(h_1/ \t Z(G))$ contains $H_1$ and thus $C_H(h_0 /\t Z(G), h_1 /\t Z(G))$ has index greater than $d$ in $C_H(h_1 /\t Z(G))$. If $C_H(h_1/ \t Z(G))$  has as well index greater than $d$ in $H$ we can iterate this process.
By the choice of $n$ and $d$ this has to stop after at most $n$ times and so we may find an element $h$ in $H \setminus \t Z(G)$ and for which the group $C_H(h/ \t Z(G))$ has index at most $d$ in $H$. As $h$ does not belong to the almost center of $G$, we have that $C_G(h)$ has infinite index in $G$ and therefore $C_H(h)$ is solvable by assumption.

Let $N$ be equal to the derived group of $\t C_H(G)$. Since it is finite and contained in $H$ it is nilpotent.
Consider the map from $C_H( h/ N)$ to $N$ sending $x$ to $[h,x]$. This map has as kernel the solvable subgroup $C_H(h)$ and as image the nilpotent group $N$. So the subgroup $C_H(h /N)$ is solvable as well. The second step is to consider the map from $C_H(h/ \t C_H(G))$ to $\t C_H(G) / N$ which maps $x$ to $[h,x] N$. Note that again the kernel $C_H(h/N)$ is solvable and the image $\t C_H(G) / N$ is abelian. So again the domain $C_H(h /\t C_H(G))$ is solvable. But this group and hence all its $H$-conjugates have finite index in $H$ and form a uniformly commensurable family of subgroups.  Applying Schlichting's theorem to this family we obtain a normal subgroup $N$ of $H$ of finite index which is a finite extension of a solvable group in a locally nilpotent. As any finite quotient of a locally nilpotent group is nilpotent, the groups $N$ as well as $H$ are solvable. This finishes the induction.

Taking any tuple $(g_1, \dots, g_n)$ in $\S$ and letting $i$ be equal to $m-1$ we obtain that $C_K(g_1)$ which equal $K$ is solvable which finishes the proof.
\qed
\begin{cor}\label{Cor_LocNilSol}
The Fitting subgroup of an $\M$-group is solvable.
\end{cor}

The proof of \cite[Lemma 3.8]{PalWag} which is stated for groups type-definable in a simple theory uses only symmetry of the almost centralizer and that they are definable. Hence it remains true for $\M$-groups.

We use the following notation: Let $L$ be group that acts on another group $A$. If $B$ is a subgroup of $A$ and $g$ an element of $L$ we denote by $C_B(g)$ the group of elements $b$ in $B$ on which $g$ acts trivially, i.\ e.\ $g  b = b$. Vice-versa, If $K$ is a subgroup of $L$ and $a$ an element of $A$, we denote by $C_K(a)$ all elements $k$ in $K$ which act trivially on $a$. 

\begin{lemma}\label{Lem_GrpAct} 
Let  $L$ and $A$ be quotients of definable subgroups of $G$ such that $A$ is abelian and $L$ acts on $A$ by conjugation (i.\ e.\ the action by $L$ on $A$ via conjugation is well-defined). Suppose $H$ is an arbitrary abelian subgroup of $L$. Assume that
\begin{itemize}
\item there are elements $ (h_i : i < l)$ in $H$ and natural numbers $(m_i : i < l)$ s.\ t.\ $$(h_i -1) ^{m_i}  A\ \mbox{\bf is finite }\ \forall i < l;$$
\item for any $h$ in $H$ the index of $C_A(h_i, \dots, h_{l-1} , h)$ in $C_A(h_i, \dots, h_{l-1})$ is finite.
\end{itemize}
Then there is a definable subgroup $K$ of $L$ which almost contains $H$ and a natural number $m$ such that $\t C_A^m(K)$ has finite index in $A$.
\end{lemma}

\proof
We denote by $\bar h$ the tuple $ (h_i : i < l)$.
As $G$ is an $\M$-group and $A$ is a definable quotient, we know that there is a natural number $d$ such that every descending chain of centralizers
$$ C_A (g_0) \leq  C_A (g_0, g_1) \leq \dots \leq  C_A (g_0, \dots, g_n) \leq \dots $$
with $g_i$ in $L$ and each of index greater than $d$ is of finite length. Hence, the group
$$ K = \{ g \in C_L(\bar h): [C_A(\bar h): C_A(\bar h, g)] < \infty\}$$
is definable and it contains $H$. Let $m$ be equal to $1+ \sum_{i =0}^{l-1} (m_i -1)$ and fix an arbitrary tuple $\bar n=(n_0, \dots, n_{m-1})$ in $l^{\times m}$. By the pigeonhole principle and the choice of $m$ there is at least one $i$ less than $l$ such that at least $m_i$ many coordinates of $\bar n$ are equal to $i$. As the group ring $\Z(H)$ is commutative and  $(h_i -1) ^{m_i}  A$ is finite for all $i$ less than $l$ by assumption, we have that
$$(h_{n_0}-1)  (h_{n_1}-1)  \ldots  (h_{n_{m-1}}-1)  A$$
is finite.

\begin{claim}
Let $g$ be in $L$ and $B$ subgroup of $A$. The set $(g -1 )B$ is finite if and only if  $B \ls C_A(g)$.
\end{claim}
\proof
Suppose that $B \not\ls C_A(g)$. Then we have a set of representatives $\{b_i : i \in \omega\}$ of cosets of $B$ modulo $ C_A(g)$, i.\ e.\ for $i$ different than $j$ we have that $b_i b_j^{-1}$ does not belong to $C_A(g)$. Thus
$$0 \neq (g-1)  b_i b_j^{-1} = ((g-1)  b_i) +((g-1)  b_j^{-1}) = ((g-1)  b_i)- ((g-1)  b_j)$$
%
Hence for $i$ different than $j$, we have that
$$( g-1)   b_i   \neq (g-1)  b_j$$
which contradicts that $(g -1 )B$ is finite.

On the other hand if  $B \ls C_A(g)$ then there exists elements $b_0, \dots, b_n$ in $B$ such that for all $b$ in $B$ there exists $i$ less or equal to $n$ such that $ b^{-1} b_i$ belongs to $ C_A(g)$, i.\ e.\ $(g-1)  b =( g-1)  b_i$. Hence the set  $(g -1 )B$ is equal to $(g-1) \{b_0, \dots, b_n\}$ and whence finite.
\qed

So, applying the claim to $(h_{n_0}-1)  (h_{n_1}-1)  \ldots  (h_{n_{m-1}}-1)  A$, for all $i\leq n$ we obtain that
$$ (h_{n_1}-1)  \ldots  (h_{n_{m-1}}-1)  A \ls C_A(h_i).$$
Thus
$$ (h_{n_1}-1)  \ldots  (h_{n_{m-1}}-1)  A \ls C_A(\bar h).$$
Since for all $k_0$ in $K$, we have that $ C_A(\bar h) \ls C_A(k_0)$, we have as well that
$$(h_{n_1}-1)  \ldots  (h_{n_{m-1}}-1)  A \ls C_A(k_0)$$
and again by the claim we deduce that
$$ (k_0 -1)  (h_{n_1}-1)  \ldots  (h_{n_{m-1}}-1)  A$$
is finite.
As $K$ is contained in the centralizer of $\bar h$, the previous line is equal to
$$ (h_{n_1}-1)  \ldots  (h_{n_{m-1}}-1)  (k_0 -1)  A.$$
Now, we repeat the previous process $m$ times and we obtain that for any $m$ tuple $(k_0, \dots k_{m-1})$ in $K$ we have that the set
$$ (k_{m-1}-1)  \ldots  (k_1-1)  (k_0 -1)  A$$
is finite. As the tuple is arbitrary, we have that for any $k$ in $K$ the group $  (k_{m-2}-1)  \ldots  (k_1-1)  (k_0 -1)  A$ is almost contained in the centralizer $C_A(k)$, i.\ e.\
$$ K \leq \t C_G(  (k_{m-2}-1)  \ldots  (k_1-1)  (k_0 -1)  A)$$
By symmetry we have that
$$ (k_{m-2}-1)  \ldots  (k_1-1)  (k_0 -1)  A \ls \t C_A(K)$$
As $L$ is an $\M$-group, we have that $ \t C_A(K)$ is definable, thus we may work modulo this group we have that
$$ (k_{m-2}-1)  \ldots  (k_1-1)  (k_0 -1)  A / \t C_A(K)$$
is finite for all choices of an $m-1$ tuple $(k_0, \dots , k_{m-2})$ in $K$. Thus as before we obtain by the claim and symmetry that
$$ (k_{m-3}-1)  \ldots  (k_1-1)  (k_0 -1)  A \ls \t C_A(K/ \t C_A(K) ) = \t C^2_A(K)$$
Repeating this process $m$ many times yields that $A \ls  \t C^m_A(K)$.
\qed

\begin{theorem}
The Fitting subgroup of an $\M$-group is nilpotent.
\end{theorem}

\proof
The Fitting subgroup $F(G)$ of $G$ is solvable by Corollary \ref{Cor_LocNilSol}. So there exists a natural number $r$ such that the $r$th derived subgroup $F(G)^{(r)}$ of $F(G)$ is trivial, hence nilpotent. Now we will show that if $F(G)^{(n+1)}$ is nilpotent, then so is $F(G)^{(n)}$. So suppose that  $F(G)^{(n+1)}$ is nilpotent. As it is additionally normal in $G$, using Corollary \ref{Cor_NorNilEnv} we can find a definable normal nilpotent subgroup $N$ of $G$ containing $F(G)^{(n+1)}$, and notice that the  central series
$$ \{1\} = N_0 < N_1 < \dots < N_k = N$$
with $N_i$ is equal to $Z_i(N)$ consist of definable normal subgroups of $G$ such that $[N, N_{i+1}] \leq N_i$.

Note first, that it is enough to show that $ F(G)^{(n)}$ is almost nilpotent: If $ F(G)^{(n)}$ is almost nilpotent it has a normal nilpotent subgroup $F$ of finite index by Lemma  \ref{Lem_DefEnvNil}. As $ F(G)^{(n)}$ is a subgroup of the Fitting subgroup, any finite subset is contained in a normal nilpotent subgroup of $G$. Thus, there is a normal nilpotent subgroup that contains a set of representatives of cosets of $F$ in $F(G)^{(n)}$. Hence the group $ F(G)^{(n)}$ is a product of two nilpotent normal subgroups and whence nilpotent. 

Second, any group almost contained in an almost nilpotent group is almost nilpotent as well. Thus, if  suffices to find a group $A$ which almost contains $ F(G)^{(n)}$ and such that $A \leq \t C_G^l(A)$ for some $l$.

As $ F(G)^{(n)}/N$ is abelian and $G/N$ is an $\M$-group, one can find a definable subgroup $A'$ of $G$ which contains $F(G)^{n}$ such that $A'/N$ is an FC-group, i.\ e.\  $A' \leq \t C_G(A'/N)$. The next step is to find a definable subgroup $A$ of $A'$ of finite index and a natural number $m$  for which $N \leq \t C_G^m(A)$. This will imply that $A \leq \t C_G(A/N) \leq \t C_G(A/\t C_G^m(A)) = \t C_G^{m+1}(A)$. As $A$ almost contains $ F(G)^{(n)}$, the group $ F(G)^{(n)}$ would be nilpotent by the above.

Fix now some $i > 0$. For any $g$ in $ F(G)^{(n)}$ there is some normal nilpotent subgroup $H_g$ which contains $g$. So $N_iH_g$ is nilpotent as well. Therefore we can find a natural number $m_g$ such that $[N_i,_{m_g} g] \leq \{1\}$ or seen as the group action as in Lemma \ref{Lem_GrpAct}
$$(g-1)^{m_g+1} N_i =\{1\}.$$
Additionally, as $G$ is an $\M$-group, we can find a finite tuple $\bar g$ in $ F(G)^{(n)}$ such that for any $g \in F(G)^{(n)}$ the index $[C_{N_i}( \bar g / N_{i-1}) : C_{N_i}( \bar g / N_{i-1}, g /N_{i-1} )]$ is finite. So we may apply Lemma \ref{Lem_GrpAct}  to $G/N$ acting on $N_i/N_{i-1}$ and the abelian subgroup $F(G)^{(n)}/N$. Thus, there is a natural number $m_i$ and a definable group $K_i$ that almost contains $F(G)^{(n)}$ such that $N_i \ls \t C_G^{m_i}(K_i /N_{i-1})$. Then the finite intersection $A= A' \cap  \bigcap_i K_i$ is a definable subgroup of $G$ which still almost contains $F(G)^{(n)}$. As for $A'$, we have that  $A \leq \t C_G(A/N)$. Additionally:
$$N_i \ls \t C_G^{m_i}(K_i /N_{i-1}) \leq  \t C_G^{m_i}(A /N_{i-1})$$
and inductively
\begin{eqnarray*}
N &\ls &\t C_G^{m_k}(A /N_{k-1}) \\
&\leq &\t C_G^{m_k}(A /(C_G^{m_{k-1}}(A /N_{k-2})))\  =\ C_G^{m_k + m_{k-1}}(A /N_{k-2})\\
&\leq & \dots \  \leq \  C_G^{m_k + \dots + m_1}(A )
\end{eqnarray*}
Using that $A \leq \t C_G(A/N)$, we obtain that $A \leq \t C_G^m(A)$ for $m  = m_k + \dots + m_1 +1$.

Overall, we get that $F(G)^{(n)}$ is nilpotent for all $n$. In particular, the Fitting subgroup $F(G)$ of $G$ is nilpotent.
\qed

\section{Almost nilpotent subgroups of $\M$-groups} \label{sec_AlNilSub}

Again, for this section we fix an $\M$-group $G$.

The goal is to generalize results on nilpotent subgroups to almost nilpotent subgroups of $G$. 

\begin{lemma}
Assume that $G$ is almost nilpotent and let $N$ be a nontrivial intersection of normal $A$-definable subgroup of $G$. Then $\[N, G\]_A$ is properly contained in $N$ and $N \cap \t Z(G)$ is a nontrivial subgroup of $G$. In particular, any minimal normal $A$-invariant subgroup of $G$ is contained in its almost center.
\end{lemma}

\proof
As $N$ is an intersection of $A$-definable normal subgroup of $G$ and we have trivially that $N\ls \t C_G(G/N)$, the group $\[N,G\]_A$ is contained in $N$. Additionally, by Lemma \ref{Lem_ComBas} we that $\[N,G\]_A$ is also contained in $\[G,G\]_A$. Inductively we obtain $(\tg_{i+1} (N, G^i))_A \leq N \cap (\tg_{i+1} G^{i+1})_A$. As $G$ is almost nilpotent $(\tg_m G^m)_A$ is trivial for some natural number $m$. Hence $(\tg (N,G))_A$ has to be properly contained in $N$ because if not $(\tg_m(N, G^{m-1}) )_A$ would be equal to $N$ as well. Additionally again by Lemma \ref{Lem_ComBas} we have that $(\tg_m (N, G^{m-1}))_A< (\tg_m G^m)_A$ and thus it is also trivial. Now choose $n$ such that $(\tg_{n+1} (N, G^{n}))_A$ is trivial and properly contained in $(\tg_{n} (N, G^{n-1}))_A$. Hence
the group $(\tg_{n} (N, G^{n-1}))_A$ is almost contained in the almost center of $G$. Since the almost center of $G$ is definable, Corollary \ref{cor_HHleqK} yields that $(\tg_{n} (N, G^{n-1}))_A$ is actually contained in $\t Z(G)$. As additionally the group $(\tg_{n} (N, G^{n-1}))_A$ is nontrivial and contained in $N$, the subgroup $N \cap \t Z(G)$ is nontrivial as well.
\qed

Now, we want to generalize a theorem due to Hall on nilpotent groups to almost nilpotent groups. To do so, we establish a connection between the triviality of the $n$th iterated almost commutator of a normal subgroup $H$ of $G$ and the almost nilpotency class of $H$. 

\subsection{Connection between the almost commutator and almost nilpotency}\label{subsec_ACAN}
\begin{lemma}\label{Lem_bra2}
 Let $H$ be an $A$-ind-definable normal subgroup of $G$ and $K$ be an intersection of $A$-definable normal subgroup of $G$. If $\[H,G\]_A \leq K$ then $H \leq \t C_G^2(G/K)$.
\end{lemma}
\proof
 Let $\[H,G\]_A \leq K$. Then by Lemma \ref{Lem_SmlSubF}, we have that $H \ls \t C_G(G/K)$. So $H/\t C_G(G/K)$ is a bounded group and as $H$ is normal in $G$, it contains the elements $h^g \cdot \t C_G(G/K)$ for all $h$ in $H$ and $g$ in $G$ . Hence the set of conjugates in $G$ of any element in $h$ in $H$ modulo $\t C_G(G/K)$ is bounded. As this corresponds to the index of the centralizer of $C_G(h / \t C_G(G/K))$ in $H$, the group $H$ is contained in the almost centralizer $\t C_G( G /\t C_G(G/K))$, i.\ e.\ the group $H$ is contained in $ \t C_G^2(G/K)$.
\qed

\begin{cor}\label{Cor_GamNTri}
If $H$ is an $A$-ind-definable subgroup of $G$ and almost nilpotent of class $n$, then $(\t \gamma_{n+1})_A H$ is trivial. Conversely, if $(\t \gamma_{n+1})_A H$ is trivial, then $H$ is almost nilpotent of class at most $n+1$.
\end{cor}

\proof
To easer notation we assume that $A$ is the empty set.

To proof the first result, we show by induction on $i$ that the almost commutator $\tg_{i+1} H$ is contained in $\t C_G^{n-i}(H)$. As  $H$ is almost nilpotent of class $n$, i.\ e.\ $ H\leq \t C_G^n(H)$, the inclusion is satisfied for $i$ equals to zero. Now suppose it holds for all natural numbers smaller than $i$. The induction hypothesis together with Lemma \ref{Lem_ComBas}(1)  implies that $\tg_{i+2} H = \[ \tg_{i+1} H, H \]$ is contained in $\[\t C_G^{n-i}(H), H\]$.
Moreover, we have that $\t C_G^{n-i}(H) = \t C_G(H/\t C_G^{n-i-1}(H))$. As the iterated centralizer $\t C_G^{n-i-1}(H)$ is definable, the definition of the almost commutator yields that $\[\t C_G^{n-i}(H), H\]$ is contained in $\t C_G^{n-i-1}(H)$. Hence $\tg_{i+2} H $ is also contained in $\t C_G^{n-i-1}(H)$ which finishes the induction. Letting $i$ be equal to $n$, we obtain that  $\tg_{n+1} H$ is contained in $\t C_G^0(H)$ which is the trivial group by definition.

For the second result, we first show by induction that for $i$ less or equal to $n-1$, we have the following inclusion:
$$\t \gamma_{(n+1)-i} (H) \leq \t C_G^{i} (H).$$
For $i=0$, the inequality holds by hypothesis. Now we assume, the inequality holds for $i < n-1 $. Thus $\tg_{(n+1)-i}(H) \leq \t C_G^{i} (H)$ or in other words  $\[\tg_{(n+1)-(i+1)} (H), H\] \leq \t C_G^{i} (H)$. By Corollary \ref{Cor_EqCenCom}, we have that $$\tg_{(n+1)-(i+1)} (H) \ls \t C_G(H/ \t C_G^i(H)) = \t C_G^{i+1}(H).$$ By Corollary \ref{cor_HHleqK}, as $(n+1)-(i+1)$ is at least $2$, finally we obtain $\tg_{(n+1)-(i+1)} (H) \leq  \t C_G^{i+1}(H)$ which finishes the induction.

Now, we let $i$ be equal to $n-1$ we obtain: $\[ H,H\] \leq \t C_G^{n-1} (H)$. Then by Lemma \ref{Lem_bra2} we have that $  H \leq \t C_G^{n+1} (H)$ and hence $H$ is almost nilpotent of class $n+1$.
\qed

\subsection{Generalized Hall's nilpotency criteria}

Now we want to prove the generalization of Hall's nilpotency criteria. 

\begin{nota}
If $K_1, \dots, K_n$ are $A$-ind-definable subgroups of $G$, we denote by $\tg_n (K_1, \dots, K_n)$ the almost commutator $\[ \dots \[\[K_0, K_1\],K_2\], \dots, K_n\t]$ and we set $\tg_1 K_1 = K_1$. If $K_i, \dots, K_{i+j-1}$ are all equal to $K$ we can replace the sequence by $K^j$.
\end{nota}

In next two lemmas are the preparation to show the approximate version of Hall's theorem.

\begin{lemma}\label{Lem_NNN}
Let $N$ be a normal $A$-ind-definable subgroup of $G$. Then,  for any natural numbers $n\geq 2$, $m$, $i$ and $j$ we have that $$\[\tg_{n} N^n, \tg_{i+j}(N^i, G^j)\] \leq \tg_{n+i} N^{n+i}.$$
\end{lemma}

\proof
If $i$ is equal to $1$, we have the following:
$$\[\[\tg_{n} N^n, \tg_{j} G^j\], N\]  \overset{\ref{Lem_ComBas}(1) +(2)}{\leq}  \[\tg_{n} N^n, N\] =  \tg_{n+1} N^{n+1}$$
$$\[\[\tg_{n} N^n, N\],  \tg_{j} G^j\]  \overset{\ref{Lem_ComBas}(2)}{\leq}  \[\tg_{n} N^n, N\] =  \tg_{n+1} N^{n+1}$$
Hence, as $ \tg_{n+1} N^{n+1}$ is the intersection of $A$-definable subgroups, the three subgroup lemma (Corollary \ref{Cor_TSL}) and  Corollary \ref{cor_HHleqK} yields that $\[\tg_{n} N^n, \[N, \tg_{j}G^j\]\]$ is contained in $\tg_{n+1} N^{n+1}$.

If $i$ is greater than $1$, we have that:
$$\[\tg_{n} N^n, \tg_{i+j}(N^i, G^j)\]  \overset{\ref{Lem_ComBas}(1) +(2)} {\leq} \[\tg_{n} N^n, \tg_{i}N^i\] $$
We apply multiple times the three subgroup lemma to this expression and obtain again by Corollary \ref{cor_HHleqK} that the commutator $ \[\tg_{n} N^n, \tg_{i}N^i\] $ is contained in $\tg_{n+i} N^{n+i}$.
\qed

 The following lemma is \cite[Lemma 7]{Hall} generalized to our framework.

\begin{lemma}\label{Lem_PrepNil}
Let $N$ be a normal $A$-ind-definable subgroup of $G$ and suppose that there exists a natural number $m$ such that $\tg_{m+1} (N, G^m )_A \ls \[ N, N\]_A$. Then, we have for all natural numbers $r$ that
$$\tg_{rm+1} (N^r ,G^{rm-r+1})_A \leq \tg_{r+1} (N^{r+1})_A.$$
\end{lemma}

\proof
In this proof we omit sometimes the $\textasciitilde$ on the commutator to make the proof more readable but all taken commutator are understood to be almost commutators.
Let $X$ be a normal ind-definable subgroup of $G$. We will prove the following by induction on $n$:
\begin{eqnarray}\label{eqn_XNGn}
\tg_{n+2} (X, N ,G^n) \leq \prod_{i=0}^n \left[\tg_{i+1}(X G^i ), \, \tg_{n-i+1} (N G^{n-i})\right]
 \end{eqnarray}
The three subgroup lemma (Corollary \ref{Cor_TSL}) insures that the equation holds for $n=1$. So assume the theorem hold for $n$. We compute:
\begin{eqnarray*}
\tg_{n+3} (X, N ,G^{n+1})& = & \left[\tg_{n+2} (X, N ,G^n), G\right]\\
& \overset{IH}{\underset {\ref{Lem_ComBas}(1)}{\leq}}&  \left[\prod_{i=0}^n \left[\tg_{i+1}(X ,G^i ), \, \tg_{n-i+1} (N, G^{n-i})\right], G\right]
 \end{eqnarray*}
As all factors are normal invariant subgroups of $G$ we may apply Lemma \ref{Lem_StabMul} finitely many times to the last expression and continue the computation:
\begin{eqnarray}\label{eqn_lala}
\ \ \ \ \ \ \ \ \ \ \ \ \ \ \ \ \ \ \  & \leq& \prod_{i=0}^n \left[  \left[ \tg_{i+1}(X ,G^i ), \, \tg_{n-i+1} (N ,G^{n-i})\right], G\right] 
\end{eqnarray}
To easer notation, we let $X_i= \tg_{i+1}(X ,G^i )$ and $N_j =  \tg_{j+1} (N ,G^{j})$.
Now, fix some $i$ less or equal to $n$. We have that 
\begin{eqnarray*}
 \left[  \left[ \tg_{i+1}(X ,G^i ), G\right],  \tg_{n-i+1} (N ,G^{n-i})\right] & = &  \left[ \tg_{i+2}(X ,G^{i+1} ), \tg_{n-i+1} (N ,G^{n-i})\right]\\
 & = &[ X_{i+1}, N_{n-i}]
 \end{eqnarray*}
and
\begin{eqnarray*}
\left[  \left[  \tg_{n-i+1} (N ,G^{n-i}), G\right], \tg_{i+1}(X ,G^i )\right] &= & \left[  \tg_{n-i+2} (N ,G^{n+1-i}), \tg_{i+1}(X ,G^i ),\right] \\
&=& [X_i, N_{n-i+1}]
\end{eqnarray*}
As the groups on the left are intersection of definable subgroups of $G$, using the approximate three subgroup lemma (Corollary \ref{Cor_TSL}), we obtain the following almost inequation for the $i$th factor of $(\ref{eqn_lala})$:
$$  \left[  \left[ \tg_{i+1}(X ,G^i ), \, \tg_{n-i+1} (N ,G^{n-i})\right], G\right] \ls   [ X_{i+1}, N_{n-i}]  \cdot   [X_i, N_{n-i+1}]$$
%
%
%
Over all, we get that
\begin{eqnarray*}
\tg_{n+3} (X, N ,G^n)& \ls & \prod_{i=0}^{n+1}  [X_i, N_{n-i+1}] \\
& = & \prod_{i=0}^{n+1} \left[ \tg_{i+1}(X, G^i ), \, \tg_{n+1-i+1} (N ,G^{n+1-i})\right]
 \end{eqnarray*}
And by Corollary \ref{cor_HHleqK} we get finally that:
\begin{eqnarray*}
\tg_{n+3} (X, N ,G^n)& \leq & \prod_{i=0}^{n+1} \left[ \tg_{i+1}(X, G^i ), \, \tg_{n+1-i+1} (N ,G^{n+1-i})\right]
 \end{eqnarray*}
Now, we prove the Lemma by induction on $r$.  By Corollary \ref{cor_HHleqK}, the almost inequality   $\tg_{m+1} (N, G^m )_A \ls \[ N, N\]_A$ implies immediately $\tg_{m+1} (N, G^m )_A \leq \[ N, N\]_A$. Thus, for $r$ equals to $1$ the lemma holds trivially by the hypothesis. Assume that the result holds for a given $r$ greater or equal to $1$. We want to prove that
$$\tg_{(r+1)m+1} (N^{r+1} ,G^{(r+1)m-r})_A \leq \tg_{r+2} (N^{r+2})_A.$$
Now consider equation (\ref{eqn_XNGn}) with $n$ replaced by $l = (r+1)m -r$ and $X$ replaced by $\tg_{r} N^r$. This gives us:
\begin{eqnarray}
 \tg_{(r+1)m+1} ( N^{r+1} ,G^{(r+1)m -r})& = & \tg_{((r+1)m -r)+2} ( \tg_r N^{r}, N ,G^{(r+1)m -r})\\
 \label{eqn_mys}
 &\leq& \prod_{i=0}^{(r+1)m -r} \left[\tg_{i+1}(\tg_{r} N^r G^i ), \, \tg_{l-i+1} (N G^{l-i})\right]
 \end{eqnarray}
The group on the left hand side is the one we want to analyze. The goal is to proof that all factors on the right hand side are contained in $\tg_{r+2} N^{r+2}$.

First let let $i$ be greater than $rm-r$.
By induction hypothesis, we have that $$\tg_{rm + 1} (N^r ,G^{rm-r+1}) \leq \tg_{r+1} (N^{r+1}).$$
As $\tg_{rm+r+1} (N^r ,G^{rm-r+1}) $ is normal in $G$ and the intersection of $A$-definable groups, using Lemma \ref{Lem_ComBas} (2) we obtain that $\tg_{r+i} (N^r ,G^{i}) \leq \tg_{r+1} (N^{r+1})$. We consider the factor for the given $i$ of equation (\ref{eqn_mys}).
\begin{eqnarray*}
 \left[\tg_{i+1}((\tg_{r} N^r) G^i ), \, \tg_{l-i+1} (N G^{l-i})\right]& \overset{\ref{Lem_ComBas}(1)}{\leq}  &\left[\tg_{r+1} N^{r+1}, \, \tg_{l-i+1} (N G^{l-i})\right] \\
& \overset{\ref{Lem_NNN}}{\leq} & \tg_{r+2} N^{r+2}
 \end{eqnarray*}
Now, let $i \leq rm-r$. By hypothesis of the lemma for $r=1$ we have that  $\tg_{m+1} (N, G^{m}) \leq \[N, N\]$. As $l -i$ is greater than $ m$ and  $\tg_{m+1} (N, G^{m})$ is an intersection of normal subgroup of $G$, we also have that $\tg_{l-i+1} (N, G^{l-i}) \leq \[N, N\]$. So we may compute:
\begin{eqnarray*}
 \left[\tg_{i+1}((\tg_{r} N^r) G^i ), \, \tg_{l-i+1} (N G^{l-i})\right]&\overset{\ref{Lem_ComBas}(1)}{\leq}   &\left[\tg_{i+r}( N^r G^i ), \, [N,N]\right] \\
& \overset{\ref{Lem_NNN}}{\leq} & \tg_{r+2} N^{r+2}
 \end{eqnarray*}
Hence all factors and whence $\tg_{(r+1)m+1} (N^{r+1}, G^{(r+1)m -r}) $ is contained in $\tg_{r+2}N^{r+2}$ which finishes the proof.
\qed

The next result is know as Hall's nilpotency criterion \cite[Theorem 7]{Hall}  if one replace almost nilpotent by nilpotent and leaving out the $A$-invariant assumption.
\begin{cor}\label{Cor_AppNil}
Let $N$ be normal $A$-ind-definable subgroup of $G$.  If $N$ is almost nilpotent of class $m$ and $G/\[N,N\]_A$ is almost nilpotent of class $n$ then $G$ is almost nilpotent of class at most ${ {m+1} \choose 2}n-{n \choose 2}+1$.
\end{cor}

\proof
By hypothesis we have that $\tg_{m+1} N^{m+1} = {1}$ and $\tg_{n+1} G^{n+1} \leq \[N,N\]$. Hence $\tg_{n+1} NG^{n} \leq \[N,N\]$ and whence $N$ satisfies the hypothesis of Lemma \ref{Lem_PrepNil}. Thus $\tg_{rn+1} (N^r ,G^{rn-r+1}) \leq \tg_{r+1} (N^{r+1})$ holds for  all natural numbers $r$. Iterating this process, we obtain that $\tg_{f(i)+1} G^{f(i)+1} \leq \tg_{i+1} N^{i+1}$ for $f(x) =$ $ {x+1} \choose 2$ $ n \ -$ $ x \choose 2$. Choosing $i$ to be $m$ we get that  $\tg_{f(m)+1} G^{f(m)+1} \leq \tg_{m+1} N^{m+1} = \{ 1 \} $. So Corollary \ref{Cor_GamNTri} yields that $G$ is almost nilpotent of class at most ${ {m+1} \choose 2}n-{n \choose 2}+1$.
\qed

\begin{cor}
Let $H$ and $K$ be normal ind-definable subgroups of $G$. If $ \[H, K\] \leq \[ K,K\]$, then for all natural numbers $r$ we have $\tg_r H^{r} \leq \tg_r K^{r}$. If $\[H,K\]$ and $\[H, H\]$ are contained in $\[K, K\]$ then $\tg_r H^{r}$ is contained in $ \tg_r K^{r}$.
\end{cor}

\proof
Let $N=H/\tg_r H^{r}$ and $M=K/\tg_r H^{r}$. Then $\[N,M\]$ is contained in $\[N,N\]$. Hence we may apply Lemma \ref{Lem_PrepNil} with $m=1$ and obtain that $$(\tg_r K^r )/\tg_r H^{r} = \tg_r M^r \leq \tg_r (N^{r-1} M) \leq  \tg_r N^{r},$$ which is trivial. Thus $\tg_r K^{r}$ is contained in $\tg_r H^{r}$. 

Consider $L= HK$. Then we can compute that $$\[ L, L\] \leq\[H, H\] \cdot \[K, K\] \cdot  \[ H,K\]=  \[K, K\].$$
By the first part of the corollary we can conclude that $\tg_r H^{r}\leq \tg_r L^{r} = \tg_r K^{r}$.
\qed

\end{document}